\documentclass[final,11pt,fleqn,leqno]{siamltex}
\usepackage[bitstream-charter]{mathdesign}
\usepackage{amsfonts}
\usepackage{amsmath}
\usepackage{tikz}
\usepackage{algorithm}
\usepackage{url}
\usetikzlibrary{positioning,shapes,arrows}
\usetikzlibrary{trees}
\def\RR{{\mathbb R}}
\def\CC{{\mathbb C}}

\newcommand{\ph}{\phantom}

\newcommand{\lsim}{\stackrel{\textstyle <}{\raisebox{-.6ex}{\(\sim\)}}}

\newtheorem{example}[theorem]{Example}

\title{Roots of bivariate polynomial systems \\
via determinantal representations}

\author{
Bor~Plestenjak\thanks{%
Department of Mathematics, University of Ljubljana, Jadranska 19,
SI-1000 Ljubljana, Slovenia, {\tt bor.plestenjak@fmf.uni-lj.si}.}
\and
Michiel~E.~Hochstenbach\thanks{%
Department of Mathematics and Computer Science,
TU Eindhoven,
PO Box 513, 5600 MB, The Netherlands.
{\tt www.win.tue.nl/$\sim$hochsten}.
This author was supported by an NWO Vidi grant.}}

\begin{document}
\maketitle

\begin{abstract}
We give two determinantal representations for a bivariate polynomial.
They may be used to compute the zeros of a system of two of these polynomials
via the eigenvalues of a two-parameter eigenvalue problem.
The first determinantal representation is suitable for polynomials with scalar or
 matrix coefficients, and consists of matrices with asymptotic order
$n^2/4$, where $n$ is the degree of the polynomial.
The second representation is useful for scalar polynomials and has
asymptotic order $n^2/6$.
The resulting method to compute the roots of a system of two bivariate polynomials
is competitive with some existing methods
for polynomials up to degree 10, as well as for polynomials with a small number of terms.
\end{abstract}



\section{Introduction}
In this paper, we make some progress on a problem that has essentially been open 
since 1902 \cite{Dixon}. It is well known that for each monic polynomial
$p(x)=p_0+p_1x+\cdots +p_{n-1}x^{n-1}+x^n$ one can construct a matrix 
$A\in\CC^{n\times n}$, such that $\det(x I- A)=p(x)$. One of the options is a 
companion matrix (see, e.g., \cite[p.~146]{HJo85})
\[A_p=\left[\begin{matrix}0 & 1 & 0 & \cdots & 0\cr
0 & 0 & 1 & \ddots & \vdots \cr
\vdots & \vdots & & \ddots & 0\cr
0 & 0 & & & 1 & \cr
-{p_0} & -{p_1} & \cdots & \cdots & -{p_{n-1}}\end{matrix}\right].\]
Thus, we can numerically compute the zeros of the
polynomial $p$ as eigenvalues of the corresponding companion
matrix $A_p$ using tools from numerical linear
algebra.
This approach is used in many numerical packages,
for instance in the \texttt{roots} command in Matlab~\cite{Matlab}.

The aim of this paper is to find a similar elegant tool for finding the zeros of
a system of two bivariate polynomials of degree $n$
\begin{equation}\label{p}
\begin{split}
p(x,y) &:= \sum_{i=0}^n \, \sum_{j=0}^{n-j} \ p_{ij} \, x^i \, y^j=0, \\
q(x,y) &:= \sum_{i=0}^n \, \sum_{j=0}^{n-j} \ q_{ij} \, x^i \, y^j=0.
\end{split}
\end{equation}
An approach analogous to the univariate case would be to construct
matrices $A_1$, $B_1$, $C_1$, $A_2$, $B_2$, and $C_2$ of size $n\times n$ such that
\begin{equation}\label{eq:detrep}
\begin{split}
\det(A_1+x B_1 +y C_1)&=p(x,y),\\[1mm]
\det(A_2+x B_2 +y C_2)&=q(x,y).
\end{split}
\end{equation}
This would give an equivalent two-parameter eigenvalue problem \cite{Atkinson}
\begin{equation}\label{eq:twopar}
\begin{split}
(A_1+x B_1 +y C_1) \, u_1&=0,\\[1mm]
(A_2+x B_2 +y C_2) \, u_2&=0
\end{split}
\end{equation}
that could be solved by the standard 
tools like the QZ algorithm, see \cite{HKP} for details.

This idea looks promising, but there are many obstacles on the way to a
working numerical algorithm that could be applied 
to a system of bivariate polynomials.
Although it is known for {\em more than a century} \cite{Dickson, Dixon,HV} that such
matrices of size $n \times n$ exist,
so far there are no efficient numerical algorithms that can construct them.
Even worse, it seems that the construction of such matrices
might be an even harder problem than finding zeros of polynomials $p$ and $q$.
There exist simple and fast constructions \cite{MuhicPlestenjak09, Quarez}
that build matrices of size ${\cal O}(n^2)$
that satisfy (\ref{eq:detrep}), where the resulting
two-parameter eigenvalue problem (\ref{eq:twopar}) is singular;
we will discuss more details in Section~\ref{sec:multipar}.
Recent results \cite{MuhicPlestenjak09} show that it is possible to solve singular
two-parameter eigenvalue problems numerically for small to medium-sized matrices.
However, the ${\cal O}(n^2)$ size of the matrices
pushes the complexity of the algorithm to the enormous 
${\cal O}(n^{12})$ and it is reported in \cite{MuhicPhd} that this approach
to compute zeros is competitive only for polynomials of degree $n<5$.

The construction of \cite{MuhicPlestenjak09} yields matrices that are of
asymptotic order $\frac{1}{2} n^2$, while those of \cite{Quarez} are of
asymptotic order $\frac{1}{4} n^2$.
In this paper we give two new representations.
The first one uses the tree structure of monomials in $x$ and $y$.
The resulting matrices are smaller than those of \cite{Quarez},
with the same asymptotic order $\frac{1}{4} n^2$.
This representation can be used for bivariate polynomials as well as
for polynomial multiparameter eigenvalue problems \cite{MuhicPlestenjakLAA};
that is, for polynomials with matrix coefficients.
The second representation is even more condensed, with asymptotic order $\frac{1}{6} n^2$,
and can be applied to scalar bivariate polynomials.
Although the size of the matrices asymptotically 
still grows quadratically with $n$,
the smaller size renders this approach attractive for polynomials of degree $n \lsim 10$,
or for larger $n$ if the polynomials have only few terms.
This already is an interesting size for a practical use and might
trigger additional interest in such methods that could culminate in
even more efficient representations. Moreover, as we will see,
for modest $n$, the order of the matrices is only roughly $2n$.
Furthermore, for polynomials of degree 3, we present a construction of matrices
of order (exactly) 3.

There are other ways
to study a system of polynomials as an eigenvalue
problems, see, e.g., \cite{Dreesen} and \cite{Stetter},
but they involve more symbolic computation. In 
\cite{Lasserre} an algorithm is proposed that only requires to
solve linear systems and check rank conditions, which are similar tools
that we use in the staircase method \cite{MuhicPlestenjak09} to solve the obtained 
singular two-parameter eigenvalue problem.
Of course, there are
many numerical methods that can be applied to systems of bivariate polynomials,
two main approaches are the homotopy continuation and the resultant method,
see, e.g., \cite{PHClab, Jonsson,  Sommese, Sturmfels, phc} and the references therein.
There are also
many methods which aim to compute only real solutions of a system of two real bivariate
polynomials, see, e.g., \cite{Noferini, VanBarel}.
We compare our method with two existing approaches, 
Mathematica's {\tt NSolve} \cite{Wolfram} and PHCpack \cite{phc} in Section~\ref{sec:numex},
and show that our approach is competitive for polynomials up to degree $\lsim 10$.

Let us mention that another advantage of writing the system of bivariate
polynomials as a two-parameter eigenvalue problem is that then we can
apply iterative subspace numerical methods such as the Jacobi--Davidson method
and compute just a small part of zeros close to a given target
$(x_0,y_0)$ \cite{HMP14}; we will not pursue this approach in this paper.

The rest of this paper is organized as follows. In Section~\ref{sec:motiv} we
give some applications where bivariate polynomial systems have to be solved.
In Section~\ref{sec:detrep} we introduce determinantal representations.
Section~\ref{sec:multipar} focuses on two-parameter eigenvalue problems.
In Section~\ref{sec:lin1} we give a determinantal representation that
is based on the ``tree'' of monomials, involves no computation,
and is suitable for both scalar and matrix polynomials.
The matrices of the resulting representation are asymptotically of order $\frac{1}{4} n^2$.
In Section~\ref{sec:lin2} we give a representation with smaller matrices,
of asymptotic order $\frac{1}{6} n^2$,
that involves just a trivial amount of numerical computation
(such as computing roots of low-degree univariate polynomials)
and can be computed very efficiently.
This representation may be used for scalar polynomials.
We end with some numerical experiments in Section~\ref{sec:numex}
and conclusions in Section \ref{sec:conc}.

\section{Motivation}\label{sec:motiv}
In delay differential equations, determining critical delays in the
case of so-called commensurate delays may lead to a problem of type \eqref{p} \cite{JHo09}.
The simplest example is of the form $x'(t) = a \, x(t) + b \, x(t-\tau) + c \, x(t-2 \tau)$,
where $\tau > 0$ is the delay; asked are values of $\tau$ that results in periodic solutions.
This yields $p$ and $q$ of degrees 2 and 3, respectively.
More delay terms with delays that are multiples of $\tau$ gives polynomials of higher degree.

Polynomial systems of form \eqref{p}
arise in numerous applications and fields, such as
signal processing \cite{BLe14, CRK09, Hat97, Van07} and robotics \cite{ZRo11}.
In computer aided design, one may be interested in the intersections
of algebraic curves, such as ellipses \cite{BGW88, KSP04, LSc02}.
In two-dimensional subspace minimization \cite{BSS88}, such as
polynomial tensor optimization, one is interested in two-dimensional
searches $\min_{\alpha, \beta} F(x + \alpha d_1 + \beta d_2)$, where $F: \RR^n \to \RR$,
$x$ is the current point, and $d_1$ and $d_2$ are search directions;
see \cite{Sor14, VanBarel} and the 
references therein.

In systems and control the first-order conditions of the $L_2$-approximation problem
of minimizing $\|h-\widetilde h\|^2 = \int_0^{\infty} |h(t)-\widetilde h(t)|^2 \, d\!t$,
for a given impulse response $h$ of degree $n$,
and $\text{degree}(\widetilde h) = \widetilde n \le n$, lead to a system of type \eqref{p} \cite{HMa96}.

When considering quadratic eigenvalue problems in numerical linear algebra,
it is of interest to determine $\text{argmin}_{\theta \in \CC} \|(\theta^2 A+\theta B + C)u\|$,
as an approximate eigenvalue for a given approximate eigenvector $u$, which gives a system of degree 3 in
the real and imaginary part of $\theta$ \cite[Sect.~2.3]{HVo03}.
Generalizations to polynomial eigenvalue problems give rise to polynomials
$p$ and $q$ of higher degree.

Also, there has been some recent interest in this problem in the context
of the {\sf chebfun2} project \cite{Noferini, chebfun}.
In {\sf chebfun2}, nonlinear real bivariate functions are approximated by bivariate polynomials,
so solving \eqref{p} is relevant for finding zeros of systems of real nonlinear bivariate functions and
for finding local extrema of such functions.

\section{Determinantal representations}\label{sec:detrep}
In this section we introduce determinantal
representations and present
some existing constructions.
The difference between what should theoretically be possible
and what can be done in practice is huge. The algorithms we propose
reduce the difference only by a small (but still significant) factor;
there seems to be plenty of room for future improvements.

We say that a bivariate polynomial $p(x,y)$ has degree $n$ if
all its monomials $p_{ij}x^iy^j$ have total degree less or equal to $n$, i.e.,
$i+j\le n$, and if at least one of the monomials has total degree equal to $n$.
We say that the square $m\times m$ matrices $A,B$, and $C$ form a \emph{determinantal representation}
of the polynomial $p$ if $\det(A+x B +y C)=p(x,y)$. As our motivation is
to use eigenvalue methods to solve polynomial systems, we will,
instead of determinantal representation, often use the term \emph{linearization} since
a determinantal representation transforms an eigenvalue problem that
involves polynomials of degree $n$ into a linear eigenvalue problem \eqref{eq:twopar}.
A definition of linearization that extends
that for the univariate case (see, e.g., \cite{Psarrakos}) is the following.

\bigskip\noindent
\begin{definition}
A linear bivariate pencil $A+x B + y C$ of size $m\times m$ is a linearization of the polynomial
$p(x,y)$ if there exist two polynomial matrices $L(x,y)$ and
$Q(x,y)$ such that $\det(L(x,y))\equiv\det(Q(x,y))\equiv 1$ and
\[L(x,y) \, (A+xB+y C) \, Q(x,y) = \left[\begin{matrix}p(x,y) & 0\cr 0 & I_{m-1}\end{matrix}\right].\]
\end{definition}

\noindent
We are interested not only in linearizations of scalar polynomials but also in linearizations
of matrix bivariate polynomials of the form (cf.~\eqref{p})
\begin{equation}\label{eq:matpol2}
P(x,y)=\sum_{i=0}^n \, \sum_{j=0}^{n-j} \, x^i \, y^j \, P_{ij},
\end{equation}
where the $P_{ij}$ are $k\times k$ matrices. In line with the above, a linear pencil $A+x B+y C$ of
matrices of size $m\times m$ presents a linearization (determinantal representation) of the matrix polynomial
$P(x,y)$ if there exist two polynomial matrices $L(x,y)$ and
$Q(x,y)$ such that $\det(L(x,y))\equiv\det(Q(x,y))\equiv 1$ and
\[L(x,y) \, (A+xB+y C) \, Q(x,y) = \left[\begin{matrix}P(x,y) & 0\cr 0 & I_{m-k}\end{matrix}\right].\]
In this case $\det(A+xB+yC)=\det(P(x,y))$. Each linearization of a matrix polynomial gives a
linearization for a scalar polynomial, as we can think of scalars as of $1\times 1$ matrices;
the opposite is not true in general.

Dixon \cite{Dixon} showed that for every scalar bivariate polynomial
$p(x,y)$ of degree $n$ there exists a determinantal representation
with symmetric matrices of size $n\times n$. Dickson~\cite{Dickson}
later showed that this result cannot be extended to general
polynomials in more than two variables, except for three
variables and polynomials of degree two and three, and four
variables and polynomials of degree two. Although they both
give constructive proofs, there
does not seem to exist an efficient numerical algorithm to
construct the determinantal
representation with matrices of size $n\times n$ for a given
bivariate polynomial of degree $n$.

In recent years, the research in determinantal representations is growing, as
determinantal representations for a particular subset of polynomials,
{\em real zero polynomials}, are related to linear matrix inequality
(LMI) constraints used in semidefinite programming SDP. For an
overview see, e.g., \cite{Netzer, VinnikovSurvey}; here we give just the essentials
for bivariate polynomials that are related to our problem.

We say that a real polynomial $p(x,y)$ satisfies the \emph{real zero condition}
with respect to $(x_0,y_0)\in\RR^2$
if for all $(x,y)\in\RR^2$ the univariate polynomial
$p_{(x,y)}(t)=p(x_0+tx,y_0+ty)$ has only real zeros.
A two-dimensional \emph{LMI set} is defined as
\[\left\{(x,y)\in\RR^2: A+ xB+y C\succeq 0\right\},\]
where $A,B$, and $C$ are symmetric matrices of size $m\times m$
and $\succeq 0$ stands
for positive semidefinite. In SDP
we are interested in convex sets ${\cal S}\subset\RR^2$ that admit an
LMI representation, i.e., ${\cal S}$ is an LMI set for certain
matrices $A,B$ and $C$. Such sets are called \emph{spectrahedra}
and Helton and Vinnikov~\cite{HV} showed that such ${\cal S}$ must be
an algebraic interior, whose minimal defining polynomial $p$
satisfies the real zero condition with respect to any point
in the interior of ${\cal S}$. Their results state that
if a polynomial $p(x,y)$ of degree $n$ satisfies real zero condition with
respect to $(x_0,y_0)$, then there exist
symmetric matrices $A,B$, and $C$ of size $n\times n$ such that
$\det(A+xB+yC)=p(x,y)$ and $A+x_0B+y_0C\succeq 0$.
Matrices $A,B$, and $C$ thus form a particular
determinantal representation for $p$.

The problem
of constructing an LMI representation with symmetric or Hermitian matrices
$A,B,$ and $C$ for a given spectrahedron ${\cal S}$
raised much more interest than the related problem of generating a determinantal
representation for a generic bivariate polynomial.
There exist procedures, which rely heavily on slow symbolic computation
or other expensive steps, that
return an LMI representation with Hermitian matrices for a given
spectrahedron, but they are not efficient enough. For instance, a method from \cite{Plaumann},
based on the proof from \cite{Dixon}, does return $n\times n$ matrices for
a polynomial of degree $n$, but the reported times (10 seconds for a
polynomial of degree 10) show that it is
much too slow for our purpose. As a first step of the above method
is to find zeros of a system of bivariate polynomials
of degree $n$ and $n-1$, this clearly can not be efficient enough for our needs.
In addition,
we are interested in determinantal representations for polynomials that do not
necessary satisfy the real zero condition.

In SDP and LMI the matrices
have to be symmetric or Hermitian, which is not required in our case.
We need a simple and fast numerical construction
of matrices that satisfy (\ref{eq:detrep}) and are as
small as possible---ideally their size should increase linearly and
not quadratically with $n$.

If we look at the available determinantal representations for
generic bivariate polynomials, we first have the
linearization by Khazanov with matrices
of size $n^2\times n^2$ \cite{Khazanov}. 
In \cite[Appendix]{MuhicPlestenjakLAA}, a smaller
linearization for bivariate matrix polynomials is given with 
block matrices of order $\frac{1}{2} n(n+1)$.
The linearization 
uses all monomials of degree up to $n-1$ and contains a direct expression for the
matrices $A$, $B$ and $C$ such that $\det(A+x B + y C)=p(x,y)$.
Similar to \cite{Khazanov}, it can be
applied to matrix polynomials. We give
an example for a general
matrix polynomial of degree $3$, from which it is possible to deduce the
construction for a generic degree. This linearization
will be superseded in Section~\ref{sec:lin1} by a more economical one.

\bigskip\noindent
\begin{example}\label{ex:lin0}
{\rm \cite[Appendix]{MuhicPlestenjakLAA}
We take a matrix bivariate polynomial of degree $3$
\[P(x,y)=P_{00}+x P_{10}+ y P_{01}+ x^2 P_{20}+ xy P_{11}+y^2 P_{02}+ x^3 P_{30}+ x^2y P_{21}+ xy^2 P_{12}+ y^3 P_{03}.\]
If $u$ is a nonzero vector, then $P(x,y)u=0$ if and only if
$(A + x B + y C) \underline u=0$,
where
\begin{equation}\label{eq:ex2}
{\footnotesize
A+ x B + y C = \left[\begin{matrix}P_{00} & P_{10} & P_{01}
  & P_{20} + x P_{30} &
    P_{11} + x P_{21} &
    P_{02} + x P_{12}  + y P_{03} \cr
  -x I_k & I_k & 0 & 0 & 0 & 0\cr
  -y I_k & 0 & I_k & 0 & 0 & 0\cr
  0 & -x I_k& 0 & I_k & 0 & 0 \cr
  0 & 0 & -x I_k& 0 & I_k & 0\cr
  0 & 0 & -y I_k & 0 & 0 & I_k\end{matrix}\right]}
\end{equation}
and
\[
\underline u=u\otimes\left[\begin{matrix}1& x& y& x^2 & xy &y^2\end{matrix}\right]^T.
\]
We have $\det(A + x B + y C)=\det(P(x,y))$ and
$A+xB+yC$ is a linearization of $P(x,y)$.
}
\end{example}

\bigskip\noindent
We remark that Quarez~\cite{Quarez} also gives explicit expressions for
determinantal representations. He is interested in symmetric representations
and is able to construct, for a bivariate polynomial of degree $n$ such that
$p(0)\ne 1$, a linearization with symmetric matrices of size
$N\times N$, where
\begin{equation}\label{eq:quarez}
N=2 \, {\lfloor {n/2}\rfloor +2 \choose 2} \approx \frac{n^2}{4}.
\end{equation}
This has asymptotically the same order as the linearization that we give in Section~\ref{sec:lin1}.
Let us also remark that in the phase, when we are solving a two-parameter eigenvalue
problem to compute the zeros of a system of two bivariate polynomials, we cannot
exploit the fact that the matrices are symmetric, so this is not important for our
application.

There are some other available tools, for instance it is possible to construct
a determinantal representation using the package NCAlgebra \cite{HMV,NCA} for
noncommutative algebra that runs in Mathematica \cite{Wolfram},
but this does not give satisfactory results for our
application as the matrices that we can construct have smaller size.

\section{Two-parameter eigenvalue problems}\label{sec:multipar}

In this section
we briefly present the two-parameter eigenvalue problem and the available numerical
methods. A motivation for the search for small determinantal representations is
that if we transform a system of bivariate polynomials into an eigenvalue problem,
then we can apply existing numerical methods for such problems.

A \emph{two-parameter eigenvalue problem} has the form \eqref{eq:twopar}
where $A_i,B_i$, and $C_i$ are given $n_i\times n_i$ complex matrices.
We are looking for $x,y\in\CC$ and nonzero vectors $u_i\in \CC^{n_i}$,
$i=1,2$, such that (\ref{eq:twopar}) is satisfied. In such case we say
that a pair $(x,y)$ is an eigenvalue and the
tensor product $u_1\otimes u_2$ is the corresponding eigenvector.
If we introduce the so-called operator determinants, the matrices
\begin{align}
  \Delta_0&=B_1\otimes C_2-C_1\otimes B_2,\nonumber\\
  \Delta_1&=C_1\otimes A_2-A_1\otimes C_2,\label{Deltaik}\\
  \Delta_2&=A_1\otimes B_2-B_1\otimes A_2\nonumber,
\end{align}
then the problem (\ref{eq:twopar}) is related to a coupled pair of
generalized eigenvalue problems
\begin{align}
  \Delta_1 \, w&=x \ \Delta_0 \, w,\nonumber\\[-1.5ex]
  \label{drugi}\\[-1.5ex]
  \Delta_2 \, w&=y \ \Delta_0 \, w\nonumber
\end{align}
for a decomposable tensor $w=u_1\otimes u_2$.
If $\Delta_0$ is nonsingular, then Atkinson~\cite{Atkinson} showed that
the solutions of (\ref{eq:twopar}) and (\ref{drugi}) agree and the
matrices $\Delta_0^{-1}\Delta_1$ and $\Delta_0^{-1}\Delta_2$ commute. In the nonsingular case
the two-parameter problem (\ref{eq:twopar}) has $n_1n_2$ eigenvalues and we
can numerically solve it with a variant of the QZ algorithm on (\ref{drugi}) from \cite{HKP}.
Ideally, if we could construct a determinantal representation with matrices $n \times n$
for a bivariate polynomial of degree $n$, this would be the method that we would apply on
the ``companion'' two-parameter eigenvalue problem to get the zeros of the polynomial system.
As matrices $\Delta_0,\Delta_1$, and $\Delta_2$ have size $n_1n_2\times n_1n_2$, the
computation of all eigenvalues of a nonsingular two-parameter eigenvalue problem
has time complexity ${\cal O}(n_1^3 \, n_2^3)$, which would lead to
${\cal O}(n^6)$ algorithm for a system of bivariate polynomials. Of course, for
this approach we need a construction of a determinantal representation
with matrices $n\times n$ that should not be more computationally expensive than
the effort to solve a two-parameter eigenvalue problem.

Unfortunately, all practical constructions for
determinantal representations (including the two presented in this paper)
return matrices that are much larger than $n\times n$. If we have a determinantal representation
with matrices larger than the degree of the polynomial, then the corresponding
two-parameter eigenvalue problem is singular, which means
that both matrix pencils (\ref{drugi}) are singular,
and we are dealing with a more difficult problem. 
There exists a numerical method from \cite{MuhicPlestenjakLAA}
that computes the regular eigenvalues of (\ref{eq:twopar}) 
from the common regular part of (\ref{drugi}).
For the generic singular case it is shown in \cite{MuhicPlestenjak09} that the regular
eigenvalues of (\ref{eq:twopar}) and (\ref{drugi}) do agree. 
For other types of singular two-parameter
eigenvalue problems the relation between the regular eigenvalues of  (\ref{eq:twopar})
and (\ref{drugi}) is not completely known, but the numerical examples
indicate that the method from \cite{MuhicPlestenjakLAA} can be successfully applied to such problems as well.
However, the numerical method, which is a variant of
a staircase algorithm \cite{Van_Dooren_singular}, has to make a lot of decisions on the numerical rank
and a single inaccurate decision can cause the method to fail.
As the size of the matrices increases, the gaps between singular values may numerically disappear and it may be difficult
to solve the problem.

This is not the only issue that prevents the use of determinantal representations
to solve a bivariate system. The algorithm for the singular two-parameter eigenvalue problems still
has complexity ${\cal O}(n_1^3 \, n_2^3)$, but the fast determinantal representations that we are
aware of return matrices of size ${\cal O}(n^2)$ instead of ${\cal O}(n)$. 
This is what pushes the overall complexity to ${\cal O}(n^{12})$ and makes this approach 
efficient only for polynomials of small degree. Nonetheless,
at complexity so high, each construction that gives a smaller determinantal representation can
make a change.
In view of this, we propose two new linearizations in the next two sections.

\section{First linearization}\label{sec:lin1}

We are interested in linearizations of the matrix polynomial
\[P(x,y)=P_{00} + x P_{10} + y P_{01} + \cdots + x^n P_{n0} + x^{n-1} y P_{n-1,1} + \cdots
+ y^n P_{0n}\]
of degree $n$, where $P_{ij}$ are square matrices. Our goal is to find square matrices $A,B$, and $C$ as small as possible such that
$\det(A+x B + y C)=\det(P(x,y))$. Also, we need a relation
that $P(x,y)u=0$ if and only if $(A+x B + y C)\underline u=0$, where $\underline u$ is
a tensor product of $u$ and a polynomial of $x$ and $y$. The linearization in this section
also applies to scalar bivariate polynomials, where all matrices are $1\times 1$ and $u=1$.

In Section~\ref{sec:detrep} we have given a linearization with block matrices of order $\frac{1}{2} n(n+1)$.
We can view this linearization in the following way. If $P(x,y)u=0$ for $u\ne 0$, then
$\det(A+x B +y C)\underline u=0$, where the vector $\underline u$ has the form
\begin{equation}\label{eq:blockvec}
\underline u  =
u\otimes \left[\begin{matrix}1&x & y&x^2&xy &y^2 &
\cdots &x^{n-1} &x^{n-2}y &\cdots &y^{n-1}\end{matrix}\right]^T.
\end{equation}
This means that $\underline u$ always begins with the initial block $u$ and then contains all blocks
of the form $x^jy^ku$ where $j+k\le n-1$. To simplify the presentation
we will usually omit $u$ when referring
to the blocks of the vector (\ref{eq:blockvec}). The blocks are ordered in the degree negative
lexicographic ordering, i.e., $x^ay^b\prec x^cy^d$ if $a+b<c+d$, or $a+b=c+d$ and $a>c$.

The above block structure of vector $\underline u$ is defined in the rows of the matrix 
from the second one to the last one (see Example \ref{ex:lin0}).
For each block $s=x^jy^k $ of (\ref{eq:blockvec}) such that $j+k\ge 1$ there always exists a preceding block $q$
of the grade $j+k-1$ such that either $s=x q$ or $s=y q$ (when $j\ge 1$ and $k\ge 1$ both options are possible).
Suppose that $s=x q$, ${\rm ind}(s)=i_s$, and ${\rm ind}(q)=i_q$, where
function ${\rm ind}$ returns the index of a block. Then the matrix $A+x B +y C$ has
block $-x I$ on position $(i_s,i_q)$ and block $I$ on position $(i_s,i_s)$. These are the only nonzero blocks
in the block row $i_s$. A similar construction with $-x I$ replaced by $-y I$ is used in the case $s=y q$.

The first block row of the matrix $A+x B +y C$ is used to represent the matrix polynomial $P(x,y)$.
One can see that
there exist linear pencils $A_{i1}+x B_{1i} + y C_{1i}$, $i=1,\ldots,m$, such that
\begin{equation}\label{eq:firstrow}
P(x,y)u = \left[\begin{matrix} A_{11}+x B_{11} +y C_{11} & A_{12}+x B_{12} +y C_{12} & \cdots &
A_{1m}+x B_{1m} +y C_{1m} \end{matrix}\right]
\underline{u},
\end{equation}
where $m=\frac{1}{2} n(n+1)$ is the number of blocks in (\ref{eq:blockvec}). The pencils in (\ref{eq:firstrow}) are
not unique. For instance, a term $x^jy^k P_{jk}$ of 
$P(x,y)$ can
be represented in one of up to the three possible ways:
\begin{enumerate}
\item[a)] if $j+k<n$, we can set $A_{1p}=P_{jk}$ where $p={\rm ind}(x^jy^k )$,
\item[b)] if $j>0$, we can set $B_{1p}=P_{jk}$ where $p={\rm ind}(x^{j-1}y^k )$,
\item[c)] if $k>0$, we can set $C_{1p}=P_{jk}$ where $p={\rm ind}(x^{j}y^{k-1} )$.
\end{enumerate}
Based on the above discussion we see that not all the blocks in
(\ref{eq:blockvec}) are needed to represent a matrix polynomial $P(x,y)$. What we need
is a minimal set of monomials $x^jy^k $, where $j+k<n$, that is
sufficient for a matrix polynomial of degree $n$. We can formulate
the problem of finding the smallest possible set for a given polynomial as a graph
problem.

\begin{figure}[!htbp]
\begin{center}
\includegraphics{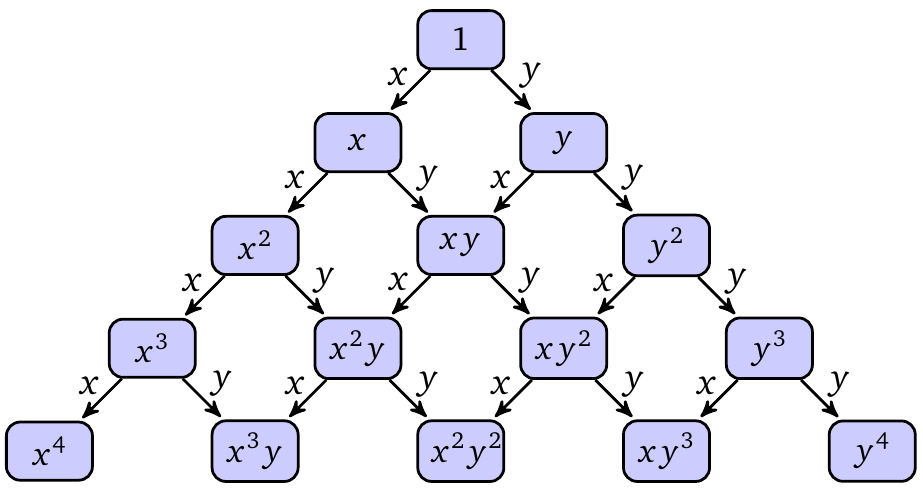}
\end{center}
\caption{Graph $G$ for the polynomial of degree $5$.\label{fig:graph}}
\end{figure}

We can think about all possible terms $x^jy^k$, where $j+k<n$, as of nodes
in a directed graph $G$ with the root $1$ and
a directed edge from node $s$ to node $t$ if
$t=x s$ or $t=y s$ 
(see Figure~\ref{fig:graph} for the case $n=5$).
Now, we are looking for the smallest 
connected subgraph $G'$ with a root $1$ that can represent a
given polynomial. 
Equivalently, we are looking for a minimal directed rooted tree.
Let us remember that for each term $x^jy^kP_{jk}$ of the
polynomial $P(x,y)$ there are up to three possible nodes in the graph $G$ that can be
used to represent it. It is sufficient that one of these nodes is in a minimal
tree $G'$. Furthermore, if $j+k>1$, then we can assume that we always
use a node of degree $j+k-1$ to represent $x^jy^kP_{jk}$ and then
there are 
only one or two options for a given term. 
All together, each nonzero term $x^jy^k P_{jk}$, where $j+k>0$, in the
polynomial $P$ defines one of the following rules for the subgraph $G'$:
\begin{itemize}
 \item[a)] if $k=0$ then 
 $x^{j-1}y^0$ has to be in the subgraph $G'$,
 \item[b)] if $j=0$ then 
 $x^{0}y^{k-1}$ has to be in the subgraph $G'$,
 \item[c)] if $j>0$ and $k>0$ then at least one of 
 $x^{j-1}y^k$ or
 $x^jy^{k-1}$ has to be in the subgraph $G'$.
\end{itemize}
The term $P_{00}$ can be presented by the root $1$,
which is always present in the subgraph $G'$.

Finding a minimal tree for a given polynomial is not an easy problem:
it can be formulated as an NP-hard directed Steiner tree problem (DST)
(see, e.g., \cite{Jungnickel}), where one has a directed graph $G=(V,E)$ with nonnegative
weights on edges and the goal is to find
the minimum weight directed rooted tree that connects all terminals $X\subset V$ to a given root
$r\in V$.

Suppose that we are looking for a minimal
representation tree for a polynomial $P(x,y)$ of degree $n$. In the graph $G$, which contains
all nodes $x^jy^k$ for $j+k<n$ (see Figure~\ref{fig:graph} for the case $n=5$), we put weight $1$
on all directed $x$ and $y$ edges. Now we add a new vertex for each monomial
$x^jy^k$ that is present in $P(x,y)$ and connect it with
zero weight edges from all possible nodes in $G$ that could be used to represent the monomial
in the linearization. We make a DST problem by taking node $1$ as a root and all newly added vertices as terminals.
From a solution of the DST problem the minimal representation tree can be recovered.
Although this is an NP-hard problem, there exist some polynomial time approximation algorithms that
give a solution close to the optimal one and could be used to construct
a small determinantal representation for a given polynomial with small number of nonzero terms.
For the latest available algorithms, see, e.g.,
\cite{Berman} and the references therein.

\bigskip\noindent
\begin{example}\label{ex:minlin}
{\rm
We are interested in a minimal tree for 
the matrix polynomial
\begin{equation}\label{eq:pol6}
P(x,y)=P_{00}+x P_{10} + y P_{01} + y^3 P_{03}+ x^2y^2 P_{22}+ x^4y P_{41}
+ xy^4 P_{14} + x^6 P_{60} +   x^4y^2 P_{42}.
\end{equation}
Nonzero terms in (\ref{eq:pol6}) define the nodes that have to be present in the minimal subgraph.
They are either strictly defined as are the nodes $1$, $y^2$, and $x^5$, or come in pairs
where at least one element of each pair has to be present in the subgraph. Such pairs
are $(x^2y,xy^2), (x^4,x^3y)$, and $(x^2y^3,xy^4)$.
The situation is presented in Figure~\ref{fig:graphex}, where nodes and pairs, such that either
left or right node has to be included, are shadowed green. The nodes of the minimal connected subgraph that
includes all required nodes are colored red.

In a DST 
formulation each green shadow presents a terminal linked by zero weight edges
to one or two nodes that are included in the region. On all other edges we put weight 1 and then
search
for the minimum weight directed rooted tree that connects all terminals to
the root $1$.

\begin{figure}[!htbp]
\begin{center}
\includegraphics{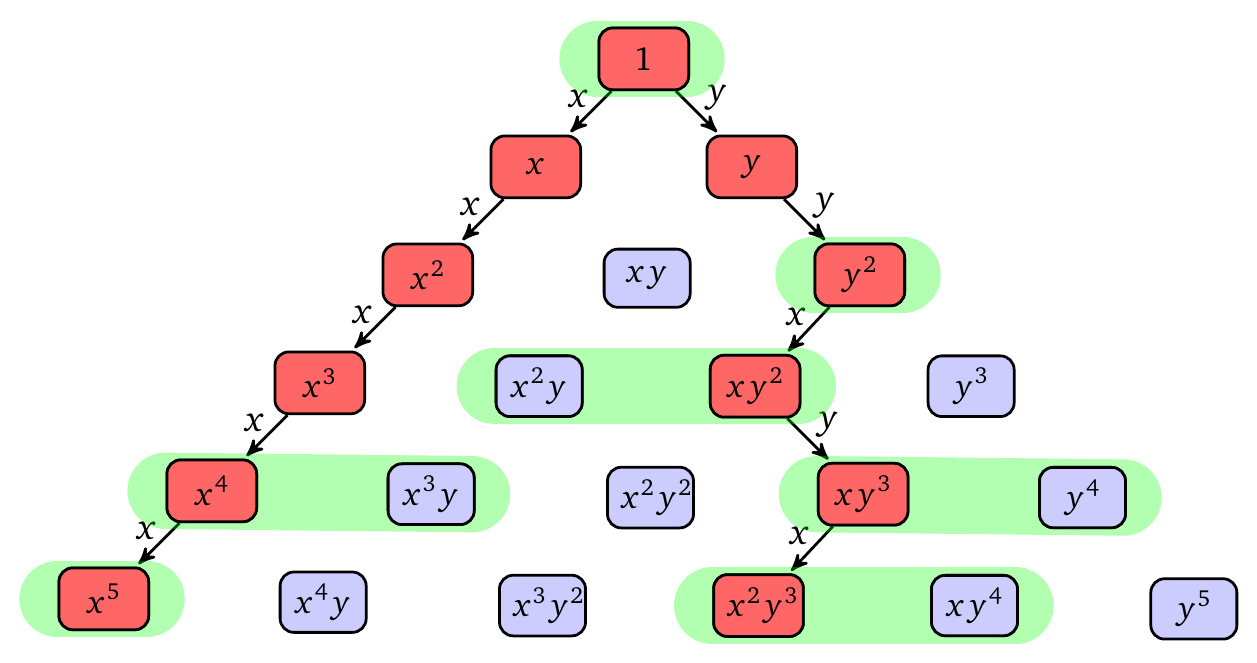}
\end{center}
\caption{A minimal directed subtree $G'$ for the matrix polynomial (\ref{eq:pol6}) of degree $6$.\label{fig:graphex}}
\end{figure}

Matrix polynomial (\ref{eq:pol6}) can thus be represented with matrices of block size $11\times 11$.
If we order the nodes of the subgraph in the degree negative lexicographic
ordering, then 
$\underline{u}$ has the form
\[\underline u  =
u\otimes \left[\begin{matrix}1&x&y&x^2&y^2 &x^3 &xy^2 & x^4 &xy^3
&x^5 &x^2y^3\end{matrix}\right]^T\]
and a possible first block row of $A+x B+y C$ has the form
\[
\left[\begin{matrix}P_{10}+x P_{10}+y P_{01} & 0 & 0 & 0 & y P_{03} & 0 & y P_{22} &
y P_{41} & y P_{14} & x P_{60} & x P_{42}\end{matrix}\right].\]
In the subsequent block rows, the matrix $A+x B+y C$ has only 20 nonzero blocks,
10 of them are identity blocks on the main diagonal. The remaining nonzero blocks
are $-x I$ on block positions $(2,1)$, $(4,2)$, $(6,4)$, $(7,5)$, $(8,6)$, $(10,8)$, $(11,9)$ and
blocks $-y I$ on positions $(3,1), (5,3), (9,7)$.
}
\end{example}

If we have a generic matrix polynomial $P(x,y)$, whose terms are all nonzero, then it is easy to see
that the 
subgraph that contains all terms 
$x^jy^k$, where $j+k<n$ and
either $k=0$ or $j$ is even,
is minimal. The detailed
situation for the case $n=6$ is presented in Figure~\ref{fig:mingraph},
and representation trees for polynomials of degree from 1 to 8 are presented in Figure~\ref{fig:lin1first8}.
Counting the number of nodes in the tree gives the following result
\begin{equation}\label{eq:sizeg}
\psi(n):=|G'|=\left\{\begin{matrix}\frac{1}{4} n(n+1), & n\ {\rm even},\cr
\frac{1}{4} (n-1)(n+5)+1, & n\ {\rm odd.}\end{matrix}\right.
\end{equation}
If we compare this with the linearization from Example~\ref{ex:lin0} that has matrices of block size $\frac{1}{2} n(n+1)$, we
see that the new linearization uses matrices of roughly half size. The
size of the matrices is also smaller than (\ref{eq:quarez}) from \cite{Quarez},
which has the same asymptotic order.

\begin{figure}[!htbp]
\begin{center}
\includegraphics{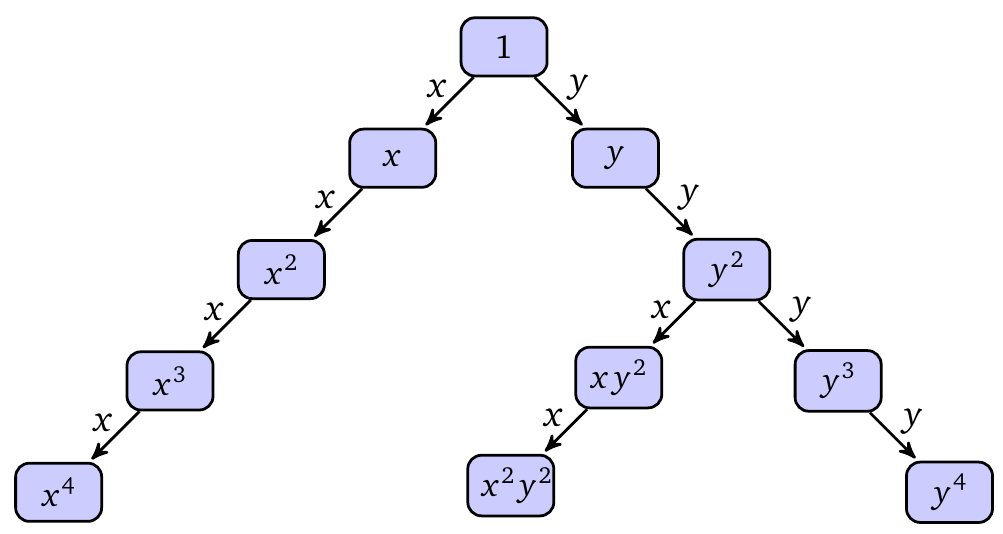}
\end{center}
\caption{A minimal tree $G'$ for a generic polynomial of degree $5$.\label{fig:mingraph}}
\end{figure}

\tikzstyle{tocka} = [circle, draw, fill=blue!40, inner sep=0pt, minimum size=5pt, node distance=0.35cm]

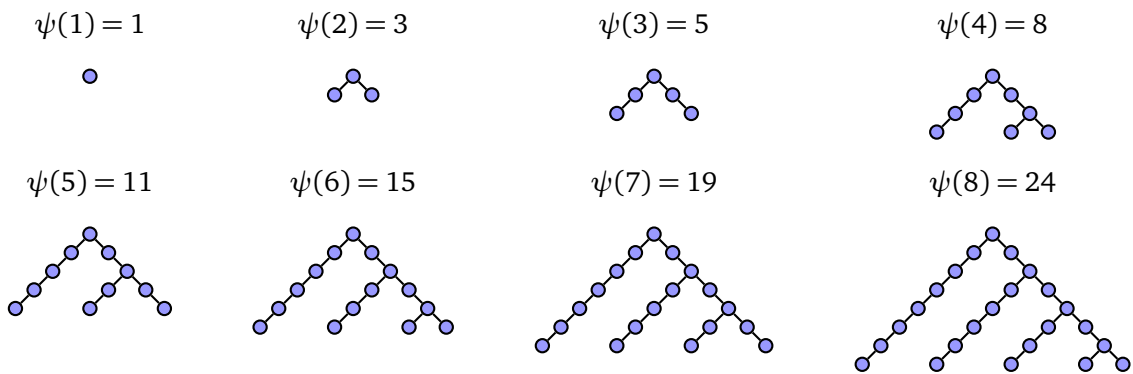
\begin{figure}[!htbp]
\begin{center}
\begin{tikzpicture}[auto,node distance=1cm,thick,main node/.style={tocka}]
  \begin{scope}[]
  \node[tocka] (00) {};
  \node [above=00,yshift=10pt] {$\psi(1)=1$};
  \end{scope}

  \begin{scope}[xshift=3.5cm]
  \node[tocka] (00) {};
  \node[tocka] (10) [below left of=00] {};
  \node[tocka] (01) [below right of=00] {};
  \node [above=00,yshift=10pt] {$\psi(2)=3$};
  \draw[-] (00) to (10);
  \draw[-] (00) to (01);
  \end{scope}

  \begin{scope}[xshift=7.5cm]
  \node[tocka] (00) {};
  \node[tocka] (10) [below left of=00] {};
  \node[tocka] (20) [below left of=10] {};
  \node[tocka] (01) [below right of=00] {};
  \node[tocka] (02) [below right of=01] {};
  \node [above=00,yshift=10pt] {$\psi(3)=5$};
  \draw[-] (00) to (01) to (02);
  \draw[-] (00) to (10) to (20);
  \end{scope}

  \begin{scope}[xshift=12cm]
  \node[tocka] (00) {};
  \node[tocka] (10) [below left of=00] {};
  \node[tocka] (20) [below left of=10] {};
  \node[tocka] (30) [below left of=20] {};
  \node[tocka] (01) [below right of=00] {};
  \node[tocka] (02) [below right of=01] {};
  \node[tocka] (03) [below right of=02] {};
  \node[tocka] (12) [below left of=02] {};
  \node [above=00,yshift=10pt] {$\psi(4)=8$};
  \draw[-] (00) to (01) to (02) to (03);
  \draw[-] (00) to (10) to (20) to (30);
  \draw[-] (02) to (12);
  \end{scope}

  \begin{scope}[yshift=-2.1cm, xshift=12cm]
  \node[tocka] (00) {};
  \node[tocka] (10) [below left of=00] {};
  \node[tocka] (20) [below left of=10] {};
  \node[tocka] (30) [below left of=20] {};
  \node[tocka] (40) [below left of=30] {};
  \node[tocka] (50) [below left of=40] {};
  \node[tocka] (60) [below left of=50] {};
  \node[tocka] (70) [below left of=60] {};
  \node[tocka] (01) [below right of=00] {};
  \node[tocka] (02) [below right of=01] {};
  \node[tocka] (12) [below left of=02] {};
  \node[tocka] (22) [below left of=12] {};
  \node[tocka] (32) [below left of=22] {};
  \node[tocka] (42) [below left of=32] {};
  \node[tocka] (52) [below left of=42] {};
  \node[tocka] (03) [below right of=02] {};
  \node[tocka] (04) [below right of=03] {};
  \node[tocka] (05) [below right of=04] {};
  \node[tocka] (06) [below right of=05] {};
  \node[tocka] (16) [below left of=06] {};
  \node[tocka] (07) [below right of=06] {};
  \node[tocka] (15) [below left of=04] {};
  \node[tocka] (25) [below left of=15] {};
  \node[tocka] (35) [below left of=25] {};
  \node [above=00,yshift=10pt] {$\psi(8)=24$};
  \draw[-] (00) to (01) to (02) to (03) to (04) to (05) to (06) to (07);
  \draw[-] (00) to (10) to (20) to (30) to (40) to (50) to (60) to (70);
  \draw[-] (02) to (12) to (22) to (32) to (42) to (52);
  \draw[-] (04) to (15) to (25) to (35);
  \draw[-] (06) to (16);
  \end{scope}

  \begin{scope}[yshift=-2.1cm, xshift=7.5cm]
  \node[tocka] (00) {};
  \node[tocka] (10) [below left of=00] {};
  \node[tocka] (20) [below left of=10] {};
  \node[tocka] (30) [below left of=20] {};
  \node[tocka] (40) [below left of=30] {};
  \node[tocka] (50) [below left of=40] {};
  \node[tocka] (60) [below left of=50] {};
  \node[tocka] (01) [below right of=00] {};
  \node[tocka] (02) [below right of=01] {};
  \node[tocka] (12) [below left of=02] {};
  \node[tocka] (22) [below left of=12] {};
  \node[tocka] (32) [below left of=22] {};
  \node[tocka] (42) [below left of=32] {};
  \node[tocka] (03) [below right of=02] {};
  \node[tocka] (04) [below right of=03] {};
  \node[tocka] (05) [below right of=04] {};
  \node[tocka] (06) [below right of=05] {};
  \node[tocka] (15) [below left of=04] {};
  \node[tocka] (25) [below left of=15] {};
  \node [above=00,yshift=10pt] {$\psi(7)=19$};
  \draw[-] (00) to (01) to (02) to (03) to (04) to (05) to (06);
  \draw[-] (00) to (10) to (20) to (30) to (40) to (50) to (60);
  \draw[-] (04) to (15) to (25);
  \draw[-] (02) to (12) to (22) to (32) to (42);
  \end{scope}

  \begin{scope}[yshift=-2.1cm, xshift=3.5cm]
  \node[tocka] (00) {};
  \node[tocka] (10) [below left of=00] {};
  \node[tocka] (20) [below left of=10] {};
  \node[tocka] (30) [below left of=20] {};
  \node[tocka] (40) [below left of=30] {};
  \node[tocka] (50) [below left of=40] {};
  \node[tocka] (01) [below right of=00] {};
  \node[tocka] (02) [below right of=01] {};
  \node[tocka] (12) [below left of=02] {};
  \node[tocka] (22) [below left of=12] {};
  \node[tocka] (32) [below left of=22] {};
  \node[tocka] (03) [below right of=02] {};
  \node[tocka] (04) [below right of=03] {};
  \node[tocka] (14) [below left of=04] {};
  \node[tocka] (05) [below right of=04] {};
  \node [above=00,yshift=10pt] {$\psi(6)=15$};
  \draw[-] (00) to (01) to (02) to (03) to (04) to (05);
  \draw[-] (00) to (10) to (20) to (30) to (40) to (50);
  \draw[-] (02) to (12) to (22) to (32);
  \draw[-] (04) to (14);
  \end{scope}

  \begin{scope}[yshift=-2.1cm]
  \node[tocka] (00) {};
  \node[tocka] (10) [below left of=00] {};
  \node[tocka] (20) [below left of=10] {};
  \node[tocka] (30) [below left of=20] {};
  \node[tocka] (40) [below left of=30] {};
  \node[tocka] (01) [below right of=00] {};
  \node[tocka] (02) [below right of=01] {};
  \node[tocka] (12) [below left of=02] {};
  \node[tocka] (22) [below left of=12] {};
  \node[tocka] (03) [below right of=02] {};
  \node[tocka] (04) [below right of=03] {};
  \node [above=00,yshift=10pt] {$\psi(5)=11$};
  \draw[-] (00) to (01) to (02) to (03) to (04);
  \draw[-] (00) to (10) to (20) to (30) to (40);
  \draw[-] (02) to (12) to (22);
  \end{scope}
\end{tikzpicture}
\vspace{-0.5em}
\end{center}
\caption{Minimal representation trees for polynomials of degrees 1 to 8.\label{fig:lin1first8}}
\end{figure}

\bigskip\noindent
\begin{theorem}\label{thm:one}
We can linearize each matrix polynomial $P(x,y)$ of degree $n$ with
matrices of block size $\psi(n)$ from (\ref{eq:sizeg}) using a minimal tree $G'$ that
contains the terms $x^jy^k$, where $j+k<n$ and either $k=0$ or $j$ is even.
\end{theorem}

\begin{proof} We order all nodes of a minimal tree $G'$ in
the degree negative lexicographic ordering and form the
block matrix $L(x,y)$ in the following way. All diagonal blocks
of $L(x,y)$ are $I$. If a node with index $p$ is connected to a node with index $q$
with an $x$ or $y$ edge, then we put $-x I$ or $-y I$
in the block position $(q,p)$, respectively. Because of the ordering, the
matrix $L$ is block lower triangular and nonsingular. Its inverse $L(x,y)^{-1}$
is therefore also a lower triangular matrix with diagonal identity blocks.

Let $m=\psi(n)$ be the number of nodes in $G'$. If follows from $L(x,y)L(x,y)^{-1}=I$ that
the first block column of $L(x,y)^{-1}$ has the form
\begin{equation}\label{eq:blockinv}
I\otimes \left[\begin{matrix}1& s_2 & s_2&  &\cdots &s_m\end{matrix}\right]^T,
\end{equation}
where $s_j$ is the monomial in the $j$th node of $G'$ for $j=1,\ldots,m$ ($s_1=1$).

Now we will construct the linearization of the matrix polynomial $P(x,y)$.
We need a block matrix $M(x,y)=A+x B+y C$, whose elements are linear pencils in $x$ and $y$.
We take $M(x,y)=L(x,y)$ and adjust the first block row $M_1(x,y)$, where
we put linear pencils such that
\[M_1(x,y)( I\otimes \left[\begin{matrix}1& s_2 & s_2&  &\cdots &s_m\end{matrix}\right]^T)=P(x,y).\]
This is always possible as for each term $x^jy^k P_{jk}$ in the polynomial $P(x,y)$
there exists a term $x^ry^q$ in $G'$ such that $(j,k)-(r,q)$ is one of the following three options: $(0,0)$, $(1,0)$, or $(0,1)$.
The product $M(x,y)L(x,y)^{-1}$ is an upper block triangular matrix of the form
\[M(x,y)L(x,y)^{-1}=\left[\begin{matrix}P(x,y) & H_2(x,y) & \cdots & H_m(x,y) \cr
 & I & &  \cr
 & & \ddots & \cr
 & & & I\end{matrix}\right],\]
 where $H_2(x,y),\ldots,H_m(x,y)$ are matrix polynomials.
 If we introduce the matrix polynomial \[U(x,y)=
 \left[\begin{matrix}I & -H_2(x,y) & \cdots & -H_m(x,y) \cr
 & I & &  \cr
 & & \ddots & \cr
 & & & I\end{matrix}\right],\]
 then it follows that
 \[U(x,y) \, M(x,y) \, L(x,y)^{-1}=
 \left[\begin{matrix}P(x,y) & & & \cr & I & &\cr & & \ddots & \cr & & & I\end{matrix}\right]\]
 and since $\det(L(x,y))\equiv \det(U(x,y))\equiv 1$,
 this proves 
 that $M(x,y)=A+x B+y C$ is indeed a linearization of the matrix polynomial $P(x,y)$.
 \end{proof}

\bigskip\noindent
\begin{example}\label{ex::cubic1}
{\rm
As an example we consider the scalar bivariate polynomial
\[p(x,y)=1+2 x+3 y+4 x^2+5xy+6y^2+7x^3+8x^2y+9xy^2+10y^3,\]
which was already linearized in \cite{MuhicPlestenjak09} with matrices of size $6\times 6$ (we can also get
a $6\times 6$ linearization if we insert the coefficients in matrix (\ref{eq:ex2}) of Example~\ref{ex:lin0}).
Now we can linearize
it with matrices of size $5\times 5$ as $p(x,y)=\det(A+x B+y C)$, where
\[
  A+ x B + y C=
  {\footnotesize\left[
  \begin{array}{cccccc}
   1+ 2x +3y & 4x + 5y & 6y & 7 x + 8 y & 9x + 10y \cr
   -x & 1 & 0 & 0 & 0 \cr
   -y & 0 & 1 & 0 & 0  \cr
   0 & -x & 0 & 1 & 0 \cr
   0 & 0 & -y & 0 & 1
  \end{array}\right]}.\]
In the next section we will further reduce the size of the matrices
to $4 \times 4$ and $3 \times 3$. 
}
\end{example}

\section{Second linearization}\label{sec:lin2}

We will upgrade the approach from the previous section and
produce even smaller representations for scalar polynomials. As before,
representations have a form of the directed tree, but instead of using
only $x$ and $y$, an edge can now be any linear polynomial
$\alpha x +\beta y$ such that $(\alpha,\beta)\ne (0,0)$. These
additional parameters give us enough freedom to produce smaller
representations. The root is still $1$ while
the other nodes are polynomials in $x$ and $y$ that are products of
all edges on the path from the root to the node.
In each node all monomials have the same degree, which
is equal to
the graph distance to the root.
Before we continue with the
construction, we give a small example to clarify the idea.

\tikzstyle{cloud3} = [ellipse, draw, fill=red!20, font={\sffamily\footnotesize},
    text width=4.2em, text centered, rounded corners, minimum width=4.1em, minimum height=1em, node distance=2cm]
\tikzstyle{cloud} = [rectangle, draw, fill=red!20, font={\sffamily\small},
    text width=1.6em, text centered, rounded corners, minimum width=1.6em, minimum height=1.4em, node distance=1.6cm]

\begin{figure}[!htbp]
\begin{center}
\includegraphics{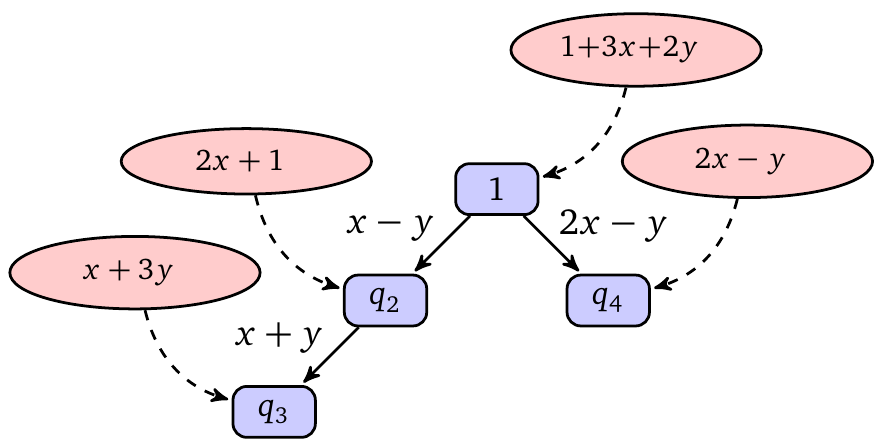}
\end{center}
\caption{A representation tree and a linearization for a polynomial of degree 3.\label{fig:ex4}}
\end{figure}

\noindent
\begin{example}\label{ex:lin2}
{\rm A linearization of a polynomial of degree $3$ with matrices of size $4\times 4$
is presented in Figure~\ref{fig:ex4}. Let us
explain the figure and show how to
produce the matrices from the representation tree.
The nodes in the representation tree are the following polynomials:
$$\begin{array}{lcl}
q_1(x,y)=1, & &
q_2(x,y)= (x-y) \, q_1(x,y) = x-y,\\[1mm]
q_3(x,y)= (x+y) \, q_2(x,y) = x^2-y^2, & &
q_4(x,y)= (2x-y) \, q_1(x,y) = 2x -y.
\end{array}
$$
The polynomial of degree $3$ is then a linear combination of
nodes in the representation tree and coefficients $f_1,\ldots,f_4$ which are polynomials
of degree $1$ contained in the ellipses. This gives
\begin{align*}
p(x,y) &= (1+3x+2y) \, q_1(x,y) + (2x+1) \, q_2(x,y) + (x+3y) \, q_3(x,y) + (2x-y) \, q_4(x,y)\\
       &=1 + 4x + y +6x^2 -6xy + y^2 + x^3 + 3x^2y - xy^2 - 3y^3.
\end{align*}

Similar as in Section~\ref{sec:lin1}, we can write the matrices by putting the linear coefficients
in the first row and relations between the polynomials $q_1(x,y)$ to $q_4(x,y)$ in the subsequent rows.
For each edge of the form $q_k=(\alpha x+\beta y) \, q_j$ we put
$-(\alpha x+\beta y)$ in the position $(k,j)$ in the matrix $M(x,y)=A+x B +y C$ and
$1$ in the position $(k,k)$.
In the first row we put $a_k+b_k x+c_k y$ in the position $(1,k)$ if
$f_k(x,y)=a_k+b_k x+c_k y$ is the linear factor that multiplies the polynomial $q_k(x,y)$
in the linearization.
The matrix $M(x,y)$ that corresponds to Figure \ref{fig:ex4}, such that $\det(M(x,y))=p(x,y)$, is
\[M(x,y)=\left[\begin{matrix} 1+3x+2y & 2x+1 & x+3y & 2x -y \cr
-x + y & 1 & 0 & 0 \cr
0 & -x -3y &  1 & 0\cr
-2x + y & 0 & 0 & 1\end{matrix}\right].\]
}
\end{example}
\bigskip

\begin{algorithm}[!htbp]
\caption{Given a bivariate 
polynomial
$p(x,y)=\alpha_{00}+\alpha_{10}x+\alpha_{01}y+\cdots + \alpha_{n0}x^n+\alpha_{n-1,1}x^{n-1}y +\cdots
+\alpha_{0n}y^n$
such that $\alpha_{n0}\ne 0$, the algorithm returns a representation tree with a
determinantal representation of the polynomial.\label{alg:lin2}}
\begin{enumerate}
\item Compute $n$ zeros $\zeta_1,\ldots,\zeta_n$ of the polynomial
$h(t)=\alpha_{n0}t^n+\alpha_{n-1,1}t^{n-1} +\cdots + \alpha_{0n}$.
\item Form a branch of the tree with the root $q_1(x,y)\equiv 1$ and
nodes $q_2,\ldots,q_n$, where $q_{k+1}$ is a successor of $q_k$ and the
edge from $q_k$ to $q_{k+1}$ contains the factor $x-\zeta_k y$ for
$k=1,\ldots,n-1$.
\item Compute linear coefficients $f_1,\ldots,f_n$ for nodes $q_1,\ldots,q_n$
in the following way:
\begin{enumerate}
\item take $f_1(x,y)=\alpha_{00}+\alpha_{10}x+\alpha_{01}y$,
\item take $f_k(x,y)=\alpha_{k0}x + (\alpha_{k-1,1}-\alpha_{k0}\beta_k)y$,
where $\beta_k$ is a coefficient of $q_k(x,y)$ at $x^{k-1}y$, for $k=2,\ldots,n-1$,
\item take $f_n(x,y)=\alpha_{n0}(x-\zeta_n y)$.
\end{enumerate}
\item Compute the remainder $r(x,y)=p(x,y)-\sum_{i=1}^n \, f_k(x,y) \, q_k(x,y),$
which has the form $r(x,y)=y^2 s(x,y)$, where $s(x,y)$ is a polynomial of
degree $n-3$.
\item If $s(x,y)\equiv 0$ then stop and return the tree.
\item Add node $q_{n+1}$ and connect it to the root by an edge having the factor $y$.
\item If $s(x,y)$ is a nonzero constant $\beta_{00}$, then use $f_{n+1}=\beta_{00}y$ as a coefficient
for the node $q_{n+1}$, stop, and return the tree.
\item Recursively call the same algorithm to obtain a representation tree with
the root $q_1'$ for
the polynomial $s(x,y)$.
\item Connect $q_{n+1}$ to $q_1'$ by an edge with a factor $y$ and return the tree
with the root $q_1$.
\end{enumerate}
\end{algorithm}

In Example \ref{ex:lin2} we showed how to construct the bivariate pencil
 $M(x,y)=A+x B + yC$ from a representation tree
and the corresponding linear coefficients. The outline of an algorithm
that constructs a representation tree and the corresponding linear coefficients
for a given polynomial $p(x,y)$
is presented in Algorithm~\ref{alg:lin2}. In the following discussion we give some missing details and
show that the algorithm indeed gives a linearization.

\begin{itemize}

\item The nodes $q_2,\ldots,q_n$ that we construct in Step 2 are polynomials of the form
$q_k(x,y)=(x-\zeta_1 y)\cdots (x-\zeta_{k-1} y)$ for $k=2,\ldots,n$.
All monomials in $q_k$ have degree $k-1$ and the leading term is $x^{k-1}$.

\item Each product $q_k(x,y)f_k(x,y)$ for $k=2,\ldots,n$ is a polynomial with monomials
of exact degree $k$, while $q_1(x,y)f_1(x,y)$ is a polynomial of degree $1$.
The linear factors $f_k(x,y)$ in Step 3 are constructed so that:
\begin{itemize}
\item leading two monomials ($x^k$ and $x^{k-1}y$) of $f_k(x,y) \, q_k(x,y)$ agree with
the part $\alpha_{k0}x^k+\alpha_{k-1,1}x^{k-1}y$ of the polynomial $p(x,y)$ for
$k=2,\ldots,n-1$,
\item the product $f_n(x,y) \, q_n(x,y)=a_{n0}(x-\zeta_1y)\cdots(x-\zeta_ny)$ agrees with
the part of $p(x,y)$ composed of all monomials of degree exactly $n$,
\item the product $q_1(x,y) \, f_1(x,y)=a_{00}+a_{10}x+a_{01}y$ agrees with the
part of $p(x,y)$ composed of all monomials of degree up to $1$.
\end{itemize}
As a result, the remainder in Step 4 has the form $y^2s(x,y)$, where $s(x,y)$ is
a polynomial of degree $n-3$. The situation at the end of Step 4 is presented in
Figure~\ref{fig:lin22}.

\begin{figure}[!htbp]
\begin{center}
\includegraphics{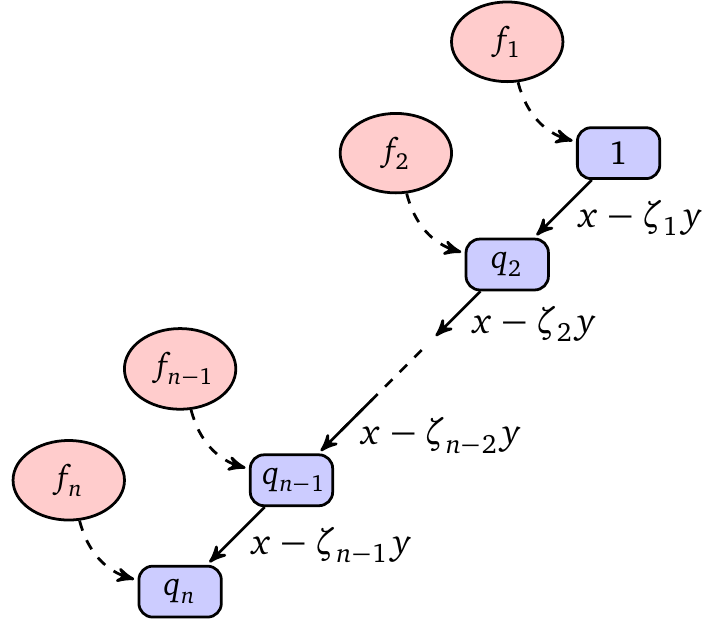}
\end{center}
\caption{The representation tree after Step 4 of Algorithm~\ref{alg:lin2}. The
remainder $p(x,y)-\sum_{j=1}^{n}f_j(x,y) \, q_j(x,y)$ is a polynomial
of the form $y^2 s(x,y)$, where $s(x,y)$ is a polynomial or degree $n-3$.\label{fig:lin22}}
\end{figure}

\item If the coefficient $\alpha_{n0}$ is zero, then we can apply a linear substitution
of $x$ and $y$ of the form $x=\widetilde x$ and $y=\widetilde y + \gamma x$, where we
pick $\gamma$ such that
\[\alpha_{n-1,1} \, \gamma + \alpha_{n-2,2} \, \gamma^2+\cdots +\alpha_{0n} \, \gamma^n\ne 0.\]
This ensures that the substituted polynomial in $\widetilde x$ and $\widetilde y$ will
have a nonzero coefficient at $\widetilde x^n$.
After we complete the representation tree for the substituted
polynomial in $\widetilde x$ and $\widetilde y$,
we perform the substitution back to $x$ and $y$.

\item If the polynomial $s(x,y)$ in Step 4 is not a constant, then we
obtain a representation subtree for $s(x,y)$
by calling recursively the same algorithm. In order to obtain the final representation tree, 
we then join the existing branch to the representation subtree for
the polynomial $s(x,y)$. We do this by introducing a new node $q_{n+1}$ in Step 6
that is linked to the root by the edge with the factor $y$.
To this new node we link the root $q_1'$ of the subtree for the polynomial $s(x,y)$ in Step 9, again using the
edge with the factor $y$. As $q_1'$ is linked to the root by two edges $y$, this multiplies all
nodes in the subtree by $y^2$ and, since the subtree is
a representation for $s(x,y)$, this gives
a representation for the remainder $r(x,y)$ from Step 4.
The situation after Step 9 with the final representation tree for the
polynomial $p(x,y)$ is presented in Figure~\ref{fig:lin23}.
\end{itemize}

\begin{figure}[!htbp]
\begin{center}
\includegraphics{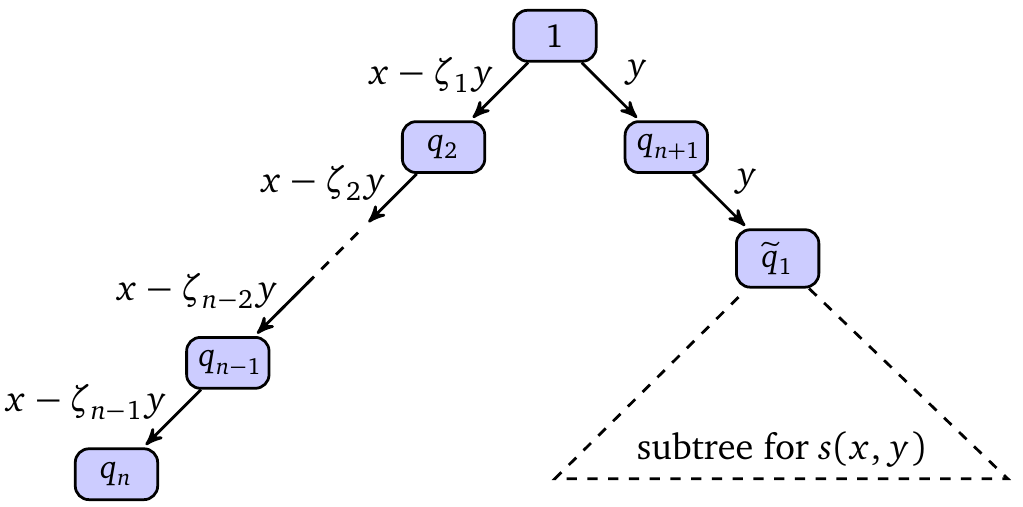}
\end{center}
\caption{The final representation tree.\label{fig:lin23}}
\end{figure}

From the output of Algorithm~\ref{alg:lin2}, matrices
$A,B,C$ such that $\det(A+xB+yC)=p(x,y)$ can be obtained in the same way as in Example~\ref{ex:lin2}.
Let us remark that the zeros $\zeta_1,\ldots,\zeta_n$ in Step 1 
can be complex, even if polynomial $p$ has real
coefficients. Thus, in a general case a linearization produced by Algorithm~\ref{alg:lin2}
has complex matrices $A,B$, and $C$.

\bigskip\noindent
\begin{example}\label{ex::pol2}
{\rm
We apply Algorithm~\ref{alg:lin2} on
$p(x,y)=1+2 x+3 y+4 x^2+5xy+6y^2+7x^3+8x^2y+9xy^2+10y^3$
from Example \ref{ex::cubic1}.
First, we compute the roots
\begin{equation}\label{eq::roots2}
\zeta_1=  -0.0079857 - 1.1259i,\ \
  \zeta_2=  -0.0079857 + 1.1259i,\ \
  \zeta_3=  -1.1269
\end{equation}
of the polynomial $h(t)=7t^3+8t^2+9t+10$.
The zeros are ordered so that $|\zeta_1|\le |\zeta_2|\le \cdots \le |\zeta_n|$.
In exact computation the order is not important, but in numerical
tests we experience better results with this order.
This gives the polynomials in the first branch of the representation tree:
\[q_1(x,y)  = 1,\quad
q_2(x,y)  = x + ( 0.0079857 + 1.1259i) y,\quad
q_3(x,y)  = x^2 + 0.015971 x y + 1.2677 y^2,\]
and we can compute the corresponding coefficients
\[f_1(x,y)  = 1 + 2x + 3y, \quad
f_2(x,y)  = 4x + (4.9681+ 4.5036i) y,\quad
f_3(x,y)  = 7x + 7.8882y.\]
For the remainder
$\displaystyle r(x,y)=p(x,y)-\sum_{j=1}^3f_j(x,y) \, q_j(x,y) = (0.88972 + 5.5576i) y^2$
we need just one additional node $q_4(x,y)=y$ with the coefficient
$f_4(x,y)=(0.88972 + 5.5576i)y$.
The determinantal representation with $4\times 4$ matrices is $p(x,y)=\det(A+xB +y C)$, where
\begin{align*}
A & = {\footnotesize \left[\begin{matrix} 1 & 0 & 0 & 0\cr
0 & 1 & 0 & 0 \cr
0 & 0 &  1 & 0\cr
0 & 0 & 0 & 1\end{matrix}\right]},\quad
B = {\footnotesize\left[\begin{array}{rrrr} 2 & 4 & 7 & 0\cr
-1 & 0 & 0 & 0 \cr
0 & -1 &  0 & 0\cr
0 & 0 & 0 & 0\end{array}\right]}, \quad \text{and} \\[1mm]
C & = {\footnotesize
\left[\begin{matrix} 3 & 4.9681+ 4.5036i & 7.8882 & 0.88972 + 5.5576i\cr
-0.0079857 + 1.1259i & 0 & 0 & 0 \cr
0 & -0.0079857 - 1.1259i &  0 & 0\cr
-1 & 0 & 0 & 0\end{matrix}\right]}.
\end{align*}
}
\end{example}
\bigskip

\tikzstyle{tocka} = [circle, draw, fill=blue!40, inner sep=0pt, minimum size=5pt, node distance=0.35cm]

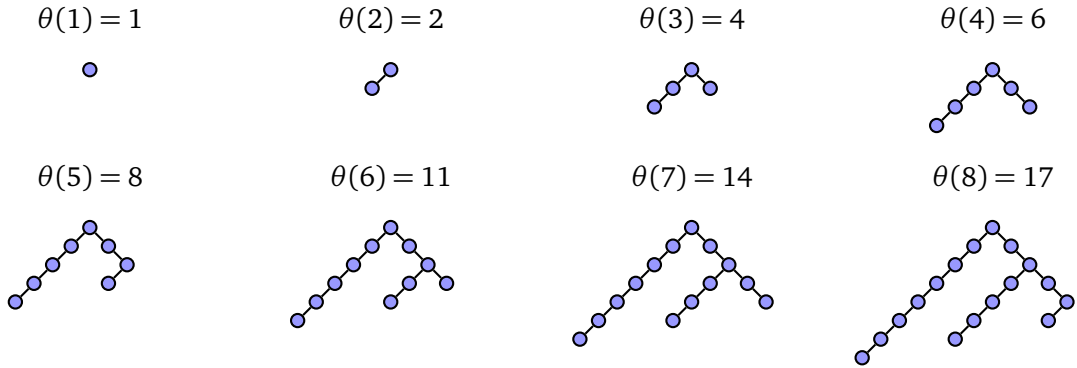
\begin{figure}[!htbp]
\begin{center}
\begin{tikzpicture}[auto,node distance=1cm,thick,main node/.style={tocka}]
  \begin{scope}[]
  \node[tocka] (00) {};
  \node [above=00,yshift=10pt] {$\theta(1)=1$};
  \end{scope}

  \begin{scope}[xshift=4cm]
  \node[tocka] (00) {};
  \node[tocka] (10) [below left of=00] {};
  \node [above=00,yshift=10pt] {$\theta(2)=2$};
  \draw[-] (00) to (10);
  \end{scope}

  \begin{scope}[xshift=8cm]
  \node[tocka] (00) {};
  \node[tocka] (10) [below left of=00] {};
  \node[tocka] (20) [below left of=10] {};
  \node[tocka] (01) [below right of=00] {};
  \node [above=00,yshift=10pt] {$\theta(3)=4$};
  \draw[-] (00) to (01);
  \draw[-] (00) to (10) to (20);
  \end{scope}

  \begin{scope}[xshift=12cm]
  \node[tocka] (00) {};
  \node[tocka] (10) [below left of=00] {};
  \node[tocka] (20) [below left of=10] {};
  \node[tocka] (30) [below left of=20] {};
  \node[tocka] (01) [below right of=00] {};
  \node[tocka] (02) [below right of=01] {};
  \node [above=00,yshift=10pt] {$\theta(4)=6$};
  \draw[-] (00) to (01) to (02);
  \draw[-] (00) to (10) to (20) to (30);
  \end{scope}

  \begin{scope}[yshift=-2.1cm, xshift=12cm]
  \node[tocka] (00) {};
  \node[tocka] (10) [below left of=00] {};
  \node[tocka] (20) [below left of=10] {};
  \node[tocka] (30) [below left of=20] {};
  \node[tocka] (40) [below left of=30] {};
  \node[tocka] (50) [below left of=40] {};
  \node[tocka] (60) [below left of=50] {};
  \node[tocka] (70) [below left of=60] {};
  \node[tocka] (01) [below right of=00] {};
  \node[tocka] (02) [below right of=01] {};
  \node[tocka] (12) [below left of=02] {};
  \node[tocka] (22) [below left of=12] {};
  \node[tocka] (32) [below left of=22] {};
  \node[tocka] (42) [below left of=32] {};
  \node[tocka] (03) [below right of=02] {};
  \node[tocka] (04) [below right of=03] {};
  \node[tocka] (15) [below left of=04] {};
  \node [above=00,yshift=10pt] {$\theta(8)=17$};
  \draw[-] (00) to (01) to (02) to (03) to (04);
  \draw[-] (00) to (10) to (20) to (30) to (40) to (50) to (60) to (70);
  \draw[-] (02) to (12) to (22) to (32) to (42);
  \draw[-] (04) to (15);
  \end{scope}

  \begin{scope}[yshift=-2.1cm, xshift=8cm]
  \node[tocka] (00) {};
  \node[tocka] (10) [below left of=00] {};
  \node[tocka] (20) [below left of=10] {};
  \node[tocka] (30) [below left of=20] {};
  \node[tocka] (40) [below left of=30] {};
  \node[tocka] (50) [below left of=40] {};
  \node[tocka] (60) [below left of=50] {};
  \node[tocka] (01) [below right of=00] {};
  \node[tocka] (02) [below right of=01] {};
  \node[tocka] (12) [below left of=02] {};
  \node[tocka] (22) [below left of=12] {};
  \node[tocka] (32) [below left of=22] {};
  \node[tocka] (03) [below right of=02] {};
  \node[tocka] (04) [below right of=03] {};
  \node [above=00,yshift=10pt] {$\theta(7)=14$};
  \draw[-] (00) to (01) to (02) to (03) to (04);
  \draw[-] (00) to (10) to (20) to (30) to (40) to (50) to (60);
  \draw[-] (02) to (12) to (22) to (32);
  \end{scope}

  \begin{scope}[yshift=-2.1cm, xshift=4cm]
  \node[tocka] (00) {};
  \node[tocka] (10) [below left of=00] {};
  \node[tocka] (20) [below left of=10] {};
  \node[tocka] (30) [below left of=20] {};
  \node[tocka] (40) [below left of=30] {};
  \node[tocka] (50) [below left of=40] {};
  \node[tocka] (01) [below right of=00] {};
  \node[tocka] (02) [below right of=01] {};
  \node[tocka] (12) [below left of=02] {};
  \node[tocka] (22) [below left of=12] {};
  \node[tocka] (03) [below right of=02] {};
  \node [above=00,yshift=10pt] {$\theta(6)=11$};
  \draw[-] (00) to (01) to (02) to (03);
  \draw[-] (00) to (10) to (20) to (30) to (40) to (50);
  \draw[-] (02) to (12) to (22);
  \end{scope}

  \begin{scope}[yshift=-2.1cm]
  \node[tocka] (00) {};
  \node[tocka] (10) [below left of=00] {};
  \node[tocka] (20) [below left of=10] {};
  \node[tocka] (30) [below left of=20] {};
  \node[tocka] (40) [below left of=30] {};
  \node[tocka] (01) [below right of=00] {};
  \node[tocka] (02) [below right of=01] {};
  \node[tocka] (12) [below left of=02] {};
  \node [above=00,yshift=10pt] {$\theta(5)=8$};
  \draw[-] (00) to (01) to (02);
  \draw[-] (00) to (10) to (20) to (30) to (40);
  \draw[-] (02) to (12);
  \end{scope}
\end{tikzpicture}
\vspace{-0.5em}
\end{center}
\caption{Representation trees for polynomials of degrees 1 to 8.\label{fig:small}}
\end{figure}

Representation trees for polynomials of degree from 1 to 8 are presented in Figure~\ref{fig:small}.
If we compare them to the determinantal representations from Section~\ref{sec:lin1} in Figure~\ref{fig:lin1first8},
then we see that representations obtained by Algorithm~\ref{alg:lin2} are much smaller.
The following lemma shows that asymptotically we use $\frac{1}{3}$ fewer nodes than in Section~\ref{sec:lin1}.

\bigskip\noindent
\begin{lemma}
Algorithm~\ref{alg:lin2} returns representation tree $G$ for the linearization of a polynomial
$p(x,y)$ of degree $n$ of size
\begin{equation}\label{eq:sizeg2}
\theta(n)=|G|=\left\{\begin{matrix}\frac{1}{6} n(n+5), & n=3k\ \ {\rm or}\ \ n=3k+1,\cr
 \frac{1}{6} n(n+5)-\frac{1}{3}, & n=3k+2.\end{matrix}\right.
\end{equation}
\end{lemma}

\begin{proof}
It follows from the recursion in the algorithm (see Figure~\ref{fig:lin23})
that the number of nodes satisfies the recurrence equation
\[\theta(n) = n + 1 + \theta(n-3).\]
The solution of this equation with the initial values $\theta(1)=1$, $\theta(2)=2$, and $\theta(3)=4$
is (\ref{eq:sizeg2}).
\end{proof}

\bigskip\noindent
For generic polynomials of degrees $n=3$ and $n=4$ it turns out to be possible
to modify the construction and save one node in the representation tree.
The main idea is to apply a linear substitution of variables $x$ and $y$ in
the preliminary phase to
the polynomial $p(x,y)$ to eliminate some of the terms.
This implies that the resulting matrices are of order $3$ ($n=3$) and $5$ ($n=4$),
instead of order $4$ and $6$ as seen before
and, it also reduces the size of the matrices for $n = 3k$ and $n = 3k+1$ by 1.
We give details in the following two subsections.

\subsection{The special case $n=3$}\label{ss:case3}
Let us consider a cubic bivariate polynomial
$p(x,y)=\alpha_{00}+\alpha_{10}x+\alpha_{01}y+\cdots+\alpha_{30}x^3+\cdots+\alpha_{03}y^3$,
where we can assume that $\alpha_{30}\ne 0$.
We introduce a linear substitution
of the form $x=\widetilde x + s\widetilde y + t$ and $y=\widetilde y$, where $s$ is
such that
\begin{equation}\label{eq:sc31}
h(s):=\alpha_{30}s^3+\alpha_{21}s^2+\alpha_{12}s + \alpha_{03}=0
\end{equation}
and $\displaystyle t={p_{20}s^2+p_{11}s+p_{02}\over h'(s)}$.

The substitution is well defined if $s$ is a single root of (\ref{eq:sc31}) and the only
situation, where this is not possible, is when $h$ has a triple root.

The above substitution transforms $p(x,y)$ into a polynomial $\widetilde p(\widetilde x,\widetilde y)$
such that its coefficients $\widetilde p_{03}$ and $\widetilde p_{02}$
are both zero. If we apply Algorithm~\ref{alg:lin2} to
$\widetilde p(\widetilde x,\widetilde y)$ and choose $\zeta_1=0$ for the first zero, then
the remainder in Step 4 is zero and we get
$3\times 3$ matrices
$\widetilde A$, $\widetilde B$, and $\widetilde C$ such that
$\det(\widetilde A+\widetilde x\widetilde B+\widetilde y \widetilde C)=\widetilde p(\widetilde x,\widetilde y)$.
Now, it is easy to see that for
$A=\widetilde A-t \widetilde B$, $B=\widetilde B$, and $C=\widetilde C-s \widetilde B$,
$\det(A+x B + y C)=p(x,y)$.

\bigskip\noindent
\begin{example}\label{ex::pol3}
{\rm
We take 
the recurrent example
$p(x,y)=1+2 x+3 y+4 x^2+5xy+6y^2+7x^3+8x^2y+9xy^2+10y^3$
(see Examples \ref{ex::cubic1} and \ref{ex::pol2}). If we take
$s=1.1269$ (see (\ref{eq:sc31})) and
$t=-0.30873$, then substitution
$x=\widetilde x + s\widetilde y + t$ and $y=\widetilde y$ changes
$p(x,y)$ into a polynomial
\[p(\widetilde x,\widetilde y)=0.55782+1.5317 \widetilde x +0.49276 \widetilde y -2.4833 \widetilde x^2 + 5.6571 \widetilde x \widetilde y
+ 7 \widetilde x^3 -15.665 \widetilde x^2\widetilde y + 17.637\widetilde x\widetilde y^2.\]
Algorithm~\ref{alg:lin2} gives $3\times 3$ matrices
$\widetilde A$, $\widetilde B$, and $\widetilde C$ such that
$\det(\widetilde A+\widetilde x\widetilde B+\widetilde y \widetilde C)=\widetilde p(\widetilde x,\widetilde y)$,
from which matrices
\begin{align*}
A & = {\footnotesize \left[\begin{matrix} 1.0307 &  -0.76665 & 2.1611\cr
  -0.30873 &    1 &            0\cr
            0  & -0.30873 &   1\end{matrix}\right],\quad
B = {\footnotesize \left[\begin{matrix} 1.5317 &  -2.4833 & 7\cr
  -1&    0 &            0\cr
            0  & -1 &   0\end{matrix}\right]},} \quad \text{and} \\[1mm]
C & = {\footnotesize \left[\begin{matrix} 2.2189 &  2.8587 & 0.00559 + 7.8813i\cr
  -1.1269 &    0 &            0\cr
            0  & -0.0079857 + 1.1259i &   0\end{matrix}\right]}
\end{align*}
such that $\det(A+x B+y C)=p(x,y)$, are obtained and we have a $3\times 3$ linearization.
}
\end{example}

\subsection{The special case $n=4$}\label{ss:case4}
 Before we give a construction for a generic quartic
bivariate polynomial, let us consider a particular case, when a
polynomial $p(x,y)=\sum_{j=0}^4\sum_{k=0}^{4-j}\alpha_{jk}x^{j}y^k$ of degree 4
is such that $\alpha_{30}=\alpha_{40}=\alpha_{03}=\alpha_{04}=0$. In this case
$5$ nodes are enough to represent the polynomial $p(x,y)$. The representation tree
for the polynomial $p(x,y)$ is presented in Figure~\ref{fig:ex4new}, where
$\zeta_1$ and $\zeta_2$ are the zeros of $\alpha_{31}\zeta^2+\alpha_{22}\zeta + \alpha_{03}$.

\begin{figure}[!htbp]
\begin{center}
\includegraphics{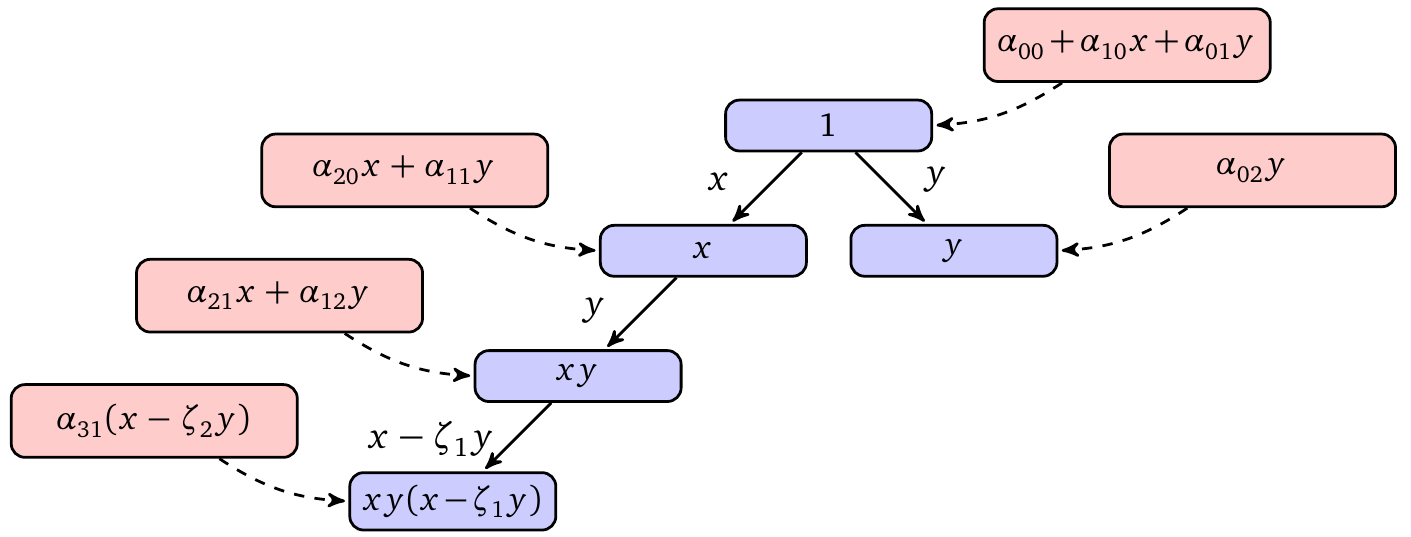}
\end{center}
\caption{A representation tree and a linearization for a polynomial
$p(x,y)=\sum_{j=0}^4\sum_{k=0}^{4-j}\alpha_{jk}x^{j}y^k$ of degree 4
such that $\alpha_{30}=\alpha_{40}=\alpha_{03}=\alpha_{04}=0$.\label{fig:ex4new}}
\end{figure}

For a generic quartic polynomial we first transform it
into one with zero coefficients at $x^3$, $x^4$, $y^3$, and $y^4$.
Except for very special polynomials, we can do this
 with a combination of two linear substitutions. Similar as in case $n=3$,
we first introduce a linear substitution
of the form $x=\widetilde x + s\widetilde y + t$ and $y=\widetilde y$, where $s$ is
such that
\[
h(s):=\alpha_{40}s^4+\alpha_{31}s^3+\alpha_{22}s^2 +\alpha_{13} s + \alpha_{04}=0
\]
and $\displaystyle t=-{\alpha_{30}s^3+\alpha_{21}s^2+\alpha_{12}s+\alpha_{03}\over h'(s)}$.

The substitution is well defined if $s$ is a single root of $h(s)$. After the
substitution we have a polynomial
$\widetilde p(\widetilde x,\widetilde y)$ such that its
coefficients $\widetilde p_{04}$ and $\widetilde p_{03}$ are both zero.
On this polynomial we apply a new substitution
$\widetilde x=\widehat x$ and $\widetilde y= u \widehat x + \widehat y + v$,
where
\[
g(u):=\widetilde \alpha_{40}+\widetilde \alpha_{31} u+\widetilde \alpha_{22} u^2 +\widetilde \alpha_{13} u^3 = 0
\]
and
$\displaystyle v=-{\alpha_{30}+\alpha_{21}u+\alpha_{12}u^2+\alpha_{03}\over g'(u)}$.

This substitution is well defined if $u$ is a single root of $g(u)$;
therefore, both substitutions exist for a generic polynomial of degree $4$.

After the second substitution we get a polynomial
$\widehat p(\widehat x,\widehat y)$ such that its
coefficients $\widehat \alpha_{30}$, $\widehat \alpha_{40}$, $\widehat \alpha_{03}$,
and $\widehat \alpha_{04}$ are all zero.
For such polynomial we can construct a representation with matrices $4\times 4$ as presented
in Figure \ref{fig:ex4new}. This gives
$4\times 4$ matrices
$\widehat A$, $\widehat B$, and $\widehat C$ such that
$\det(\widehat A+\widehat x\widehat B+\widehat y \widehat C)=\widehat p(\widehat x,\widehat y)$.
Finally, if we take
\[
A=\widehat A-t \widehat B -(v-tu) \, \widehat C,\quad
B=\widehat B-u \, \widehat C,\quad
C=\widehat C-s \, \widehat B,
\]
then $\det(A + x \, B + y \, C)=p(x,y)$.

If we add the constructions from Subsections \ref{ss:case3} and \ref{ss:case4}
as special cases to Algorithm~\ref{alg:lin2},
then we save one node for all generic
polynomials of degree $n=3k$ or $n=3k+1$. Although this advantage seems to be modest,
numerical results in the following section point out that for small $n$
this does speed up the computation of the zeros considerably
(for instance, for $n=6$ the corresponding $\Delta$-matrices are of order
$10^2 = 100$ instead of $11^2 = 121$).

\section{Numerical examples}\label{sec:numex}

Determinantal representations from Sections \ref{sec:lin1} and \ref{sec:lin2} can be used
to numerically solve a system of two bivariate polynomials. We
first linearize the problem as a two-parameter
eigenvalue problem and then solve it with the method for
singular two-parameter eigenvalue problems from \cite{MuhicPlestenjakLAA}.
Subsequently, we refine the solutions by two steps of Newton's method. We
refer to the numerical methods that use linearizations
from Sections \ref{sec:lin1} and \ref{sec:lin2} as
{\tt Lin1} and {\tt Lin2}, respectively.
In the first example we take polynomials with random coefficients,
while the second example considers some challenging benchmark polynomials.

\bigskip\noindent
\begin{example}\label{ex:randtest}\rm
We compare {\tt Lin1} and {\tt Lin2} to
{\tt NSolve} in Mathematica 9 and {\tt PHCLab} 1.02 \cite{PHClab} running
PHCpack 2.3.84 on systems of full bivariate polynomials of the same degree,
whose coefficients are random real numbers
uniformly distributed on $[0,1]$ or random complex numbers,
such that real and imaginary part are both uniformly distributed on $[0,1]$.

The results are presented in Table \ref{tbl:one}. For each $n$ we run all methods on the same set of
20 random polynomial systems and measure the average time.
{\tt Lin1}
and Mathematica's {\tt NSolve} work faster for polynomials with real coefficients
while this does not make a change for {\tt Lin2} and {\tt PHCLab}, therefore, the results in the table for {\tt Lin2} and
{\tt PHCLab} are an average of 20 real and 20 complex examples. Clearly,
if {\tt Lin1} is applied to a polynomial with real coefficients,
then matrices $\Delta_0$, $\Delta_1$, and $\Delta_2$
are real. If we apply {\tt Lin2} then the matrices are complex in general
as roots of univariate polynomials are used in the construction. Although the
complex arithmetic
is more expensive than the real one, complex eigenproblems from {\tt Lin2} are so
small that they are solved faster than the larger real problems from {\tt Lin1}.

\begin{table}[!htbp]
\begin{footnotesize}
\begin{center}
\caption{Average computational times for {\tt Lin1}, {\tt Lin2}, {\tt NSolve}, and {\tt PHCLab} for
random full bivariate polynomial systems of degree $3$ to $10$. For {\tt Lin1} and {\tt NSolve}
results are separated for real $(\RR)$ and complex polynomials $(\CC)$.
Notice that these are the running times; the accuracy of the methods
varies, as we discuss in the text.}\label{tbl:one}
\begin{tabular}{r|cccccc|rr}
\hline & \multicolumn{6}{|c|}{Time (sec)} & \multicolumn{2}{|c}{$\Delta$-matrix size}\\
$n$ & {\tt Lin1} ($\RR$) & {\tt Lin1} ($\CC$) & {\tt Lin2} & {\tt PHCLab} & {\tt NSolve} ($\CC$) & {\tt NSolve} ($\RR$) & {\tt Lin1} & {\tt Lin2}\\
\hline \rule{0pt}{2.3ex}%
3 & 0.01 & \ph20.01 & $<0.01$ & 0.18 & 0.25 & 0.04 & 25  &  9  \\
4 & 0.02 & \ph20.02 & 0.01 & 0.21 & 0.42 & 0.07 & 64  & 25  \\
5 & 0.03 & \ph20.05 & 0.02 & 0.26 & 0.67 & 0.17 & 121 & 64  \\
6 & 0.08 & \ph20.18 & 0.05 & 0.34 & 1.04 & 0.22 & 225 & 100 \\
7 & 0.23 & \ph20.57 & 0.16 & 0.44 & 2.75 & 0.61 & 361 & 169 \\
8 & 0.67 & \ph22.04 & 0.54 & 0.59 & 2.17 & 0.88 & 576 & 289 \\
9 & 2.25 & \ph26.21 & 1.33 & 0.80 & 5.53 & 1.48 & 841 & 400 \\
10 & 6.24 & 16.6\ph2 & 3.38 & 1.05 & 8.12 & 3.85 & 1225 & 576 \\
\hline
\end{tabular}
\end{center}
\end{footnotesize}
\end{table}

Computational times for {\tt Lin1}, {\tt Lin2}, and PHCLab are very similar for each of
the 20 test problems of the same degree. On the other hand, {\tt NSolve} needs
substantially more time for certain problems. For example, for complex polynomials
of degree $n=7$, {\tt NSolve} needed approximately $1.5s$ for 16 of the 20 examples,
and $7.5s$ for the additional 4 examples. That explains why the average
time for {\tt NSolve} ($\CC$) is larger in case $n=7$ than in case $n=8$.

Beside the computational time, accuracy and reliability are another important factors.
{\tt NSolve} is the only method that finds all solutions in all examples, but on
the other hand, the results are on average less accurate than with other methods.
As a measure of accuracy we use the maximum value of
\begin{equation}
\max(|p_1(x_0,y_0)|,|p_2(x_0,y_0)|) \ \|J^{-1}(x_0,y_0)\|,\label{eq:condn}
\end{equation}
where $J(x_0,y_0)$ is a Jacobian matrix of $p_1$ and $p_2$ at $(x_0,y_0)$,
over all computed zeros $(x_0,y_0)$. $\|J^{-1}(x_0,y_0)\|$ is an absolute
condition number of a zero $(x_0,y_0)$ and we assume that in random
examples all zeros are simple. For a good method (\ref{eq:condn}) should
be as small as possible.

For real or complex systems of degree $n\le 7$, {\tt Lin2} is the fastest
method and usually  the most accurate one. It is never
significantly less accurate than the others, so it clearly wins
in this case.
For real polynomials of degree $n=8$ computational times of all methods are very close. {\tt NSolve} is
the fastest method with a tight margin in 17 out of 20
cases, but is also several orders of magnitude less accurate. {\tt PHCLab} is faster
and slightly less accurate than {\tt Lin2} in 5 out of 40 cases,
but, in one of them it fails to compute all the solutions.

For $n=9$ and $n=10$ {\tt PHCLab} becomes the fastest method, but is less reliable.
In many cases it does not compute all the solutions. For $n=9$ this
happens in 14 out of 40 times and for $n=10$ in $17$ out of $40$ cases. As {\tt PHCLab}
is using random initial systems, a possible remedy is to run {\tt PHCLab} several times.
Also {\tt Lin2} fails to compute solutions for $2$ real examples for $n=10$.
A remedy for {\tt Lin2} in these cases is to interchange variables $x$ and $y$.

{\tt Lin1} is competitive in particular for real systems of degree $n\le 7$.
For $n=8$ it misses one solution in one example and in two examples for
$n=10$ we have to adapt the criteria for detecting a numerical rank
in the staircase algorithm to get the
correct number of solutions.

Let us remark that each node less in the representation tree really does make a difference.
For instance, if we do not apply the special case for $n=4$ in Subsection \ref{ss:case4}, then
the $\Delta$ matrices for {\tt Lin2} for polynomial systems of degree $n=10$ are of size $625\times 625$ instead
of $576\times 576$ and the average computational time rises from
$3.38s$ to $3.95s$.
\end{example}

\bigskip\noindent
\begin{example}\label{ex:25}\rm
We test {\tt Lin1} and {\tt Lin2} on 25 examples {\tt ex001} to {\tt ex025} from \cite{Buse}. This set
contains challenging benchmark problems with polynomials of small degree from $(3,2)$ to
$(11,10)$ that have many multiple zeros and usually have less solutions than a generic pair of the
same degrees. {\tt Lin1} and {\tt Lin2} performed satisfactorily on most examples, but, they also failed on some.
Instead of giving the details for all 25 examples, we give the key observations.
\begin{itemize}
\item Multiple zeros can present a problem for the algorithm from \cite{HKP}
that is used to solve the  projected regular problem
$\widetilde\Delta_1 \, w=x \, \widetilde \Delta_0 \, w$,
$\widetilde\Delta_2 \, w=y \, \widetilde \Delta_0 \, w$
that is obtained from (\ref{drugi}) by the modified staircase algorithm from
\cite{MuhicPlestenjakLAA}. The QZ algorithm is first applied
to 
$\widetilde\Delta_1 \, w=x \, \widetilde \Delta_0 \, w$
and then 
$\widetilde\Delta_2 \, w=y \, \widetilde \Delta_0 \, w$ is multiplied
by $Q$ and $Z$. The eigenvalues are clustered along the diagonal so that
multiple eigenvalues $x$ should be in the same block. For several of the 25 examples
with eigenvalues of high multiplicity the
clustering criteria have to be adapted otherwise the results are not so accurate.
\item {\tt Lin2} is faster, but the
accuracy can be lost if the polynomial
in Step 1 of Algorithm~\ref{alg:lin2} has multiple zeros,
an example is $p_2$ from  {\tt ex008}. The method
fails for {\tt ex014}, {\tt ex018}, and {\tt ex020}.

\item We get very good results in example {\tt ex005} with the system
$x^9+y^9-1=0$ and $x^{10}+y^{10}-1=0$ using {\tt Lin2}. In this case {\tt Lin2} returns
optimal determinantal representations with matrices of size $9\times 9$ and
$10\times 10$, respectively. The obtained two-parameter eigenvalue problem is
not singular and we get the solutions in $0.08s$, while {\tt PHCLab}
and {\tt NSolve} need $0.6s$. For comparison, {\tt Lin1}, applied to the same problem,
returns $\Delta$ matrices of size $1015\times 1015$, while
{\tt Lin2} gives $\Delta$ matrices of size $90\times 90$.
\item {\tt Lin1} is slower but can be more accurate. Because
there is no computation in the construction, no errors are introduced
in the construction of the linearization. {\tt Lin1} manages to solve 22 out of
25 examples (in some examples the parameters have to be adapted to make it work),
but fails for {\tt ex007}, {\tt ex016}, and {\tt ex018}.
\item {\tt NSolve} always finds all solutions but is slower than {\tt Lin1} and {\tt Lin2} except
for {\tt ex014} where polynomials are of degrees $11$ and $10$. {\tt PHCLab} usually finds just one instance of a multiple
eigenvalue and thus returns much less zeros.
\end{itemize}
\end{example}

\begin{example}\rm\label{ex:tri} Encouraged by the good results for
{\tt ex005} in Example \ref{ex:25}, we carry out some experiments with polynomials
of form $p(x,y)=\alpha_{n0}x^n+\cdots+\alpha_{0n}y^n+h(x,y)$, where
$h(x,y)$ is a polynomial of small degree $m\ll n$. For such polynomials Algorithm~\ref{alg:lin2}
returns matrices of size $n+1+\theta(m)$ or even smaller.
For example, it is easy to see that for $m=1$
we get linearization of the smallest possible size $n\times n$. We compared {\tt Lin2}, {\tt PHCLab}, and
{\tt NSolve}. Computational times for random polynomials with complex coefficients of the above form
are presented in Table \ref{tbl:two}. As $n$ increases, {\tt PHCLab} becomes faster then {\tt Lin2},
but in most cases it does not compute all solutions. For instance, for $n=30$ it
misses 16 and 20 zeros for $m=1$ and $m=3$, respectively. Therefore, {\tt Lin2}
might be the preferred method for such polynomial systems.

\begin{table}[!htbp]
\begin{footnotesize}
\begin{center}
\caption{Computational times for {\tt Lin2}, {\tt PHCLab}, and {\tt NSolve} for
systems of two polynomials of the form
$p(x,y)=\alpha_{n0}x^n+\cdots+\alpha_{0n}y^n+h(x,y)$,
where degree of $h(x,y)$ is $m\ll n$.}\label{tbl:two}
\vspace{3mm}
\begin{tabular}{cc|lcc}
\hline \multicolumn{2}{c}{}& \multicolumn{3}{|c}{Time (sec)}\\
$n$ & $m$ & {\tt Lin2} & {\tt PHCLab} & {\tt NSolve} \\
\hline \rule{0pt}{2.3ex}%
15 & 1 & \ph10.48 & \ph11.9 & \ph13.7 \\
15 & 3 & \ph10.87 & \ph11.8 & \ph14.1 \\[0.5mm]
20 & 1 & \ph12.1 & \ph14.5 & 10.8 \\
20 & 3 & \ph13.8 & \ph15.0 & 11.3 \\[0.5mm]
25 & 1 & \ph19.6 & 10.6 & 25.0 \\
25 & 3 & 13.1  & 12.7 & 26.3 \\[0.5mm]
30 & 1 & 23.2 &  17.4 & 52.8 \\
30 & 3 & 37.2  & 20.1 & 55.0 \\
\hline
\end{tabular}
\end{center}
\end{footnotesize}
\end{table}

\end{example}

\section{Conclusions}\label{sec:conc}

We have proposed two linearizations for bivariate polynomials. The first linearization does not involve any
computation as the coefficients of the polynomials appear as (block) coefficients
of the matrices $A$, $B$, and $C$. This linearization is suitable for both scalar and matrix bivariate polynomials.
The second linearization, useful for scalar polynomials, involves little computation and returns much
smaller matrices. They are
still larger than the theoretically smallest possible size $n\times n$, but their construction
is very simple and fast. Moreover, while the asymptotic order is $\frac{1}{6}n^2$,
the order for small $n$ is about $2n$; for polynomials of degree 3 and 4
we have presented determinantal representations of order 3 and 5, respectively.

As an application we have presented a method for finding roots of two bivariate polynomials.
We show that an approach, where the polynomial system is first linearized into a two-parameter
eigenvalue problem, which is later solved by a modified staircase method, is
numerically feasible and gives good results for polynomials of degree $n \lsim 10$,
as well as for polynomials of higher degree but with few terms.
Any further results on even smaller determinantal representations that can be
efficiently constructed numerically, could enlarge the above degree.

\bigskip\noindent{\bf Acknowledgment } The research was performed while
the first author was visiting the CASA group at TU Eindhoven.
The author wishes to thank the NWO for the visitor's grant and the CASA group for the
hospitality.

\end{document}